\documentclass[12pt]{amsart}
\usepackage{}

\usepackage{amsmath}
\usepackage{amsfonts}
\usepackage{amssymb}
\usepackage[all]{xy}           

\usepackage{bbding}
\usepackage{txfonts}
\usepackage{amscd}

\usepackage[shortlabels]{enumitem}
\usepackage{ifpdf}
\ifpdf
  \usepackage[colorlinks,final,backref=page,hyperindex]{hyperref}
\else
  \usepackage[colorlinks,final,backref=page,hyperindex]{hyperref}
\fi
\usepackage{tikz}
\usepackage[active]{srcltx}

\makeatletter

\newtheorem{thm}{Theorem}[section]
\newtheorem{lem}[thm]{Lemma}
\newtheorem{cor}[thm]{Corollary}
\newtheorem{pro}[thm]{Proposition}
\newtheorem{ex}[thm]{Example}

\newtheorem{defi}[thm]{Definition}

\newcommand {\emptycomment}[1]{}

\newcommand{\be }{\begin{equation}}
\newcommand{\ee }{\end{equation}}

\newcommand{\g}{\mathfrak g}
\newcommand{\h}{\mathfrak h}

\newcommand{\huaA}{\mathcal{A}}
\newcommand{\huaL}{\mathcal{L}}
\newcommand{\huaR}{\mathcal{R}}

\newcommand{\huaG}{\mathcal{G}}

\newcommand{\huaV}{\mathcal{V}}
\newcommand{\huaW}{\mathcal{W}}

\newcommand{\huaD}{\mathcal{D}}

\newcommand{\huaO}{{\mathcal{O}}}

\newcommand{\frkd}{\mathfrak d}
\newcommand{\frke}{\mathfrak e}

\newcommand{\frkg}{\mathfrak g}

\newcommand{\frkl}{\mathfrak l}

\newcommand{\frks}{\mathfrak s}

\newcommand{\br}[1]{   [ \cdot,    \cdot  ]   }

\newcommand{\dM}{\mathrm{d}}

\newcommand{\Hom}{\mathrm{Hom}}

\newcommand{\gl}{\mathfrak {gl}}

\newcommand{\End}{\mathrm{End}}
\newcommand{\ad}{\mathrm{ad}}

\newcommand{\K}{\mathbb{K}}

\begin{document}

\title[Pre-Lie 2-bialgebras and 2-grade classical Yang-Baxter equations]{Pre-Lie 2-bialgebras and 2-grade classical Yang-Baxter equations}

\author{Jiefeng Liu}
\address{School of Mathematics and Statistics, Northeast Normal University, Changchun 130024, China}
\email{liujf534@nenu.edu.cn}

\author{Tongtong Yue}
\address{School of Mathematics and Statistics, Northeast Normal University, Changchun 130024, China}
\email{1834955099@qq.com}

\author{Qi Wang}
\address{School of Mathematics and Statistics, Changchun University of Technology, Changchun 130012, Jilin, China}
\email{wangqi2018@ccut.edu.cn}

\vspace{-5mm}


\begin{abstract}
We introduce a notion of a para-K\"{a}hler strict Lie 2-algebra, which can be viewed as a categorification of a para-K\"{a}hler Lie algebra. In order to study para-K\"{a}hler strict Lie 2-algebra in terms of strict pre-Lie 2-algebras, we introduce the Manin triples, matched pairs and bialgebra theory for strict pre-Lie 2-algebras and the equivalent relationships between them are also established. By means of the cohomology theory of Lie 2-algebras, we study the coboundary strict pre-Lie 2-algebras and introduce 2-graded classical Yang-Baxter equations in strict pre-Lie 2-algebras. The solutions of the 2-graded classical Yang-Baxter equations are useful to construct strict pre-Lie 2-algebras and para-K\"{a}hler strict Lie 2-algebras. In particular, there is a natural construction of strict pre-Lie 2-bialgebras from the strict pre-Lie 2-algebras.
\end{abstract}


\maketitle

\tableofcontents

\allowdisplaybreaks


\section{Introduction}
The purpose of this paper is to study the concepts of para-K\"{a}hler strict Lie 2-algebras and strict pre-Lie 2-bialgebras, which both provide certain categorifications of the concepts of para-K\"{a}hler Lie algebras and pre-Lie bialgebras, respectively. The study of the coboundary strict pre-Lie 2-bialgebra leads to the 
2-graded classical Yang-Baxter equations on strict pre-Lie 2-algebras, whose solutions give a natural method to construct strict pre-Lie 2-bialgebras and para-K\"{a}hler strict Lie 2-algebras.

\subsection{Pre-Lie algebras and pre-Lie bialgebras}
Pre-Lie algebras (or left-symmetric algebras) are a class of nonassociative algebras coming from the study of convex homogeneous cones, affine manifolds and affine structures on Lie groups, deformation of associative algebras and then  appeared in many fields in mathematics and mathematical physics, such as complex and symplectic structures on Lie groups and Lie algebras (\cite{BBM,symplectic Lie algebras,Kan}),  Poisson brackets and infinite dimensional Lie algebras (\cite{BaN}), vertex algebras (\cite{Bakalov}), $F$-manifold algebras (\cite{Dot19,LSB}), homotopy algebra structures (\cite{Ban,DSV}) and operads (\cite{CL}). See  the survey articles \cite{Bai21,Burde}
and the references therein for more details. 

The concept of phase space associated with a Lie algebra was initially introduced by Kupershmidt in \cite{Ku94}. Kupershmidt demonstrated that pre-Lie algebras constitute the fundamental structures underlying the phase spaces of Lie algebras, forming a natural category that holds significant importance in both classical and quantum mechanics (\cite{Ku99a}). Furthermore, phase spaces can be precisely characterized as para-K\"{a}hler structures on Lie algebras. From a geometric perspective, a para-K\"{a}hler manifold is defined as a symplectic manifold equipped with a pair of transversal Lagrangian foliations. In particular, a para-K\"{a}hler Lie algebra corresponds to the Lie algebra of a Lie group $G$ endowed with a  $G$-invariant para-K\"{a}hler structure (\cite{Kan}). See \cite{BeM,Kan,Ku94,Ku99a} for more details about para-K\"{a}hler Lie algebras and applications in mathematical physics. Motivated by the study of para-K\"{a}hler Lie algebras and phase spaces in terms of pre-Lie algebras,   the concepts of pre-Lie bialgebras (also referred to as left-symmetric bialgebras), matched pairs, and Manin triples for pre-Lie algebras were systematically developed in \cite{Left-symmetric bialgebras}.   The notion of coboundary pre-Lie bialgebra leads to the classical Yang-Baxter equation (CYBE) on pre-Lie algebra.  A  symmetric  solution of  this  equation  gives a pre-Lie bialgebra and a para-K\"{a}hler Lie algebra naturally. A Hessian structure on a pre-Lie algebra, which corresponds to an affine Lie group $G$ with a $G$-invariant Hessian metric (\cite{Shima}), gives a non-degenerate symmetric $\frks$-matrix. In \cite{Ku99b}, Kupershmidt introduced the notion of an $\huaO$-operator on a Lie algebra in order to better understand the relationship between the classical Yang-Baxter equation and the related integrable systems. A symmetric solution of the CYBE can be equivalently described by an  $\huaO$-operator on a Lie algebra operator with respect to the coregular representation. The relationships among these mathematical structures can be summarized in the following diagram:

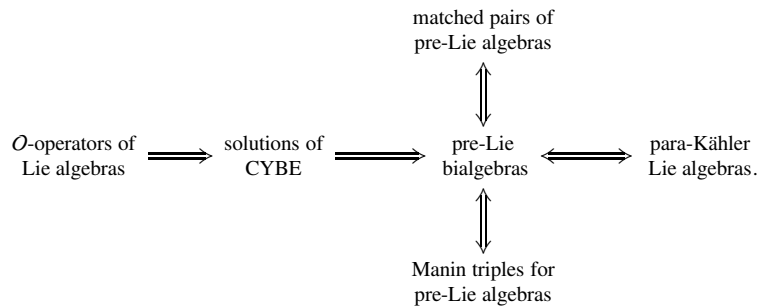
\vspace{-.1cm}
\begin{equation*}
    \begin{split}
        \xymatrix{
            &&\text{matched pairs of}\atop \text{pre-Lie algebras}&& \\
             \text{$\huaO$-operators of}\atop \text{Lie algebras}   \ar@2{->}[r]    & 
             \text{solutions of}\atop \text{CYBE} 
            \ar@2{->}[r]& \text{pre-Lie}\atop \text{ bialgebras}
            \ar@2{<->}[r] \ar@2{<->}[d] \ar@2{<->}[u]& \text{para-K\"{a}hler}\atop \text{ Lie algebras}.\\
         &   &\text{Manin triples for}\atop \text{pre-Lie algebras} && }
    \end{split}
\end{equation*}

\subsection{Lie 2-algebras and Lie 2-bialgebras}
As categorification of Lie algebras, the notion of Lie $2$-algebras was introduced by Baez and Crans in \cite{baez:2algebras}, which is just the 2-term $L_\infty$-algebra. The concept of an $L_\infty$-algebra (sometimes called a strongly homotopy (sh) Lie algebra)  was originally
introduced in \cite{LS,LM} as a model for ``Lie algebras that satisfy Jacobi identity up to all
higher homotopies''. The structure of a Lie 2-algebra also appears in
string theory, higher symplectic geometry \cite{baez:classicalstring,baez:string}, and Courant algebroids
\cite{Roytenbergphdthesis,shengzhu1}. 

A Lie bialgebra \cite{D} is the Lie-theoretic case of a bialgebra, which is a set with a Lie algebra structure and a Lie coalgebra one  which are compatible. Lie bialgebras are the infinitesimal objects of Poisson-Lie groups. Both Lie bialgebras and Poisson-Lie groups are considered as semiclassical limits of quantum groups. In order to give a model for the categorification of Lie bialgebras, Kravchenko in \cite{olga} used higher derived brackets to define $L_\infty$-bialgebra and thus obtained 2-term $L_\infty$-bialgebra. However, in this setting, although a 2-term $L_\infty$-bialgebra gives a Lie 2-algebra structure on the graded vector $V$, it does not give a Lie 2-algebra structure on the dual, $V^*$. In \cite{BSZ}, the authors adapted the shifting trick to introduce another categorification of Lie bialgebras, named by Lie 2-bialgebra, which consists of  Lie 2-algebra structures on $V$ and $V^*$, along with some compatibility conditions between them. In particular, for the strict case, they described the compatibility conditions between brackets and cobrackets as a cocycle condition. Guided by the classical theory of Lie bialgebras, they developed, in explicit terms, various higher corresponding objects, like matched pairs, Manin triples, Manin triples and their relations. Furthermore, they considered the coboundary Lie 2-bialgebras and give the notion of 2-graded classical Yang-Baxter equations (2-graded CYBE) whose solutions provide examples of Lie
2-bialgebras.

\subsection{Pre-Lie 2-algebras,  para-K\"{a}hler Lie 2-algebras  and pre-Lie 2-bialgebras}
In \cite{CL}, the authors studied pre-Lie algebras using the theory of operads, and introduced the notion of a pre-Lie$_\infty$-algebra. The author also proved that the PreLie operad is Koszul. In \cite{Sheng19}, the author introduced the notion of a pre-Lie 2-algebra, which is a categorification of a pre-Lie algebra, and proved that the category of pre-Lie 2-algebras and the category of 2-term pre-Lie$_\infty$-algebras are equivalent. Furthermore, the solutions of 2-graded classical Yang-Baxter equations are constructed by pre-Lie 2-algebras. Also, $\mathcal O$-operators on Lie 2-algebras are introduced, which can be used to construct pre-Lie 2-algebras.

Recall that a symplectic Lie algebra is a Lie algebra $\g$ equipped with a nondegenerate skew-symmetric 2-cocycle $\omega$, that is, 
$$\omega([x,y]_\g,z)+\omega([z,x]_\g,y)+\omega([y,z]_\g,x)=0,\quad\forall~ x,y,z\in \g.$$
Then, in \cite{symplectic Lie algebras}, Chu showed that the underlying algebraic structure of the symplectic Lie algebra is a pre-Lie algebra, in which the pre-Lie operation ``$\cdot_\g$" is defined by
$$\omega(x\cdot_\g y,z)=-\omega(y,[x,z]_\g),\quad~\forall x,y,z\in\g.$$
As a categorification of symplectic Lie algebras , we introduce the notion of a symplectic strict Lie 2-algebra, which consists of a  strict Lie 2-algebra equipped with a graded skew-symmetric nondegenerate closed $2$-form $\omega=(\omega_1,\omega_2)$ with $\omega_1\in \Hom(\wedge^2\g_0,\K)$ and $\omega_2\in \Hom(\g_0\wedge \g_{-1},\K)$. Furthermore, we show that the underlying algebraic structure of a symplectic strict Lie 2-algebra is a strict pre-Lie $2$-algebra.  The following commutative diagram is established:
$$
\xymatrix{
\mbox{symplectic~strict~Lie~2~algebras}
	\ar[r]^{}& \mbox{strict~pre-Lie~2-algebras}\\
	\mbox{symplectic~Lie~algebra}\ar[u]_{\mbox{categorification}}\ar[r]& \mbox{pre-Lie~algebras.}\ar[u]_{\mbox{categorification}}}
$$
In order to give a categorification of para-K\"{a}hler Lie algebras, we introduce a notion of a para-K\"{a}hler strict Lie 2-algebra, which is  a symplectic strict  Lie 2-algebra decomposed as  a direct sum of the underlying graded vector spaces of two Lagrangian strict sub-Lie $2$-algebras. Unlike the para-K\"{a}hler Lie algebras can be equivalently described by  pre-Lie bialgebras,  para-K\"{a}hler strict Lie 2-algebras can not be described by strict pre-Lie 2-bialgebras. However,   strict pre-Lie 2-bialgebras corresponds to a class of para-K\"{a}hler strict Lie 2-algebras, which is a  para-K\"{a}hler strict Lie 2-algebra with $\omega_1=0$, named by  a special para-K\"{a}hler strict Lie 2-algebra. Guided by the philosophy in \cite{BSZ}, we  introduce matched pairs, Manin triples, bialgebras for strict pre-Lie 2-algebra. The equivalences between those structures are established. We also study the cobounary strict pre-Lie 2-bialgebras, which leads to the 2-graded classical Yang-Baxter equations in strict pre-Lie 2-algebras. A solution of the 2-graded classical Yang-Baxter equations (CYBEs) gives rise to a strict pre-Lie 2-bialgebra naturally and thus give rise to a special para-K\"{a}hler strict Lie 2-algebra. The 2-graded CYBEs is interpreted in terms of  $\huaO$-operators on the subadjacent strict Lie 2-algebras. Conversely, an $\huaO$-operator on a strict Lie 2-algebra can  provide  a solution of the 2-graded CYBEs in certain bigger strict pre-Lie 2-algebras. These connections can be illustrated by the following diagram:

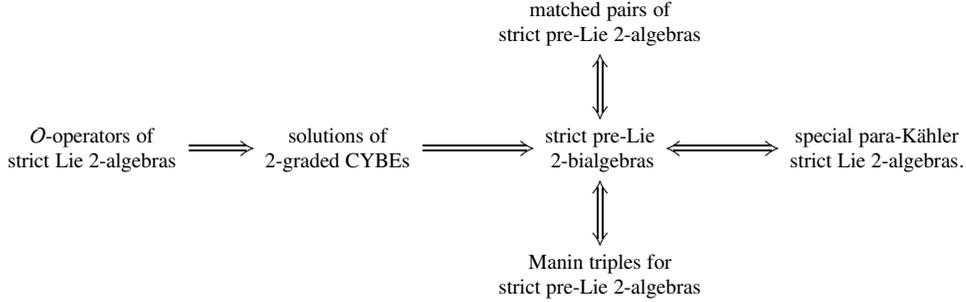
\vspace{-.1cm}
\begin{equation*}
	\begin{split}
		\xymatrix{
			&&\text{matched pairs of}\atop \text{strict pre-Lie 2-algebras}&& \\
			\text{$\huaO$-operators of}\atop \text{strict Lie 2-algebras}   \ar@2{->}[r]    & 
			\text{solutions of}\atop \text{2-graded CYBEs} 
			\ar@2{->}[r]& \text{strict pre-Lie}\atop \text{ 2-bialgebras}
			\ar@2{<->}[r] \ar@2{<->}[d] \ar@2{<->}[u]& \text{special para-K\"{a}hler}\atop \text{strict  Lie 2-algebras}.\\
			&   &\text{Manin triples for}\atop \text{strict pre-Lie 2-algebras} && }
	\end{split}
\end{equation*}
In particular, for a strict pre-Lie $2$-algebra $\huaA=(A_0,A_{-1},\dM,\cdot)$, 
 \begin{equation*}
 	R=\displaystyle\sum_{i=1}^{k}(e_i\otimes e_i^*+e_i^*\otimes e_i)+\displaystyle\sum_{j=1}^{l}(\frke_j\otimes \frke_j^*+\frke_j^*\otimes \frke_j)
 \end{equation*}
 is a solution of $2$-graded classical Yang-Baxter Equations in $\huaA\ltimes_{L^*,0}\huaA^*$ with $L^*=(L_0^*, L_1^*)$, where  $\{e_i\}_{1\leq i\leq k}$ and $\{\frke_j\}_{1\leq j\leq l}$ are the basis of $A_0$ and $A_{-1}$ respectively, and $\{e_i^*\}_{1\leq i\leq k}$ and $\{\frke_j^*\}_{1\leq j\leq l}$ are the dual basis. Besides, partial results for strict pre-Lie 2-algebras can be generalized to non-strict pre-Lie 2-algebras. We will study them in our future work.

\subsection{Outline of the paper.} The paper is organized as follows. In Section \ref{sec:Pre}, we first recall the representations and cohomology of pre-Lie algebras. Then we recall Lie 2-algebras and their representations and cohomology. In Section \ref{sec:Pre-Lie 2-algebras}, as a categorification of a symplectic structure on a Lie algebra, we give the notion of a symplectic structure on a strict Lie 2-algebra and show that a symplectic strict Lie 2-algebra can give rise to a strict pre-Lie 2-algebra. We also give the constructions of strict pre-Lie 2-algebras through $\huaO$-operators on strict Lie 2-algebras, Rota-Baxter operators of weight $\lambda$ on strict associative 2-algebras and derivations on strict commutative associative 2-algebras. Also, the representations of strict pre-Lie 2-algebras are discussed.  In Section \ref{sec:Manin triple}, we introduce the notion of a para-K\"{a}hler strict Lie 2-algebra, which is a categorification of para-K\"{a}hler Lie algebras. In order to study the para-K\"{a}hler strict Lie 2-algebras in terms of strict pre-Lie 2-algebras, a notion of special para-K\"{a}hler strict Lie 2-algebras is introduced. We introduce the Manin triples  and matched pairs of strict pre-Lie 2-algebras and show that special para-K\"{a}hler strict Lie 2-algebras, Manin triples  and matched pairs of strict pre-Lie 2-algebras are equivalent. In Section \ref{sec:Pre-Lie 2-bialgebras}, we introduce the notion of a strict pre-Lie 2-bialgebra and show that it is equivalent to matched pairs of strict pre-Lie 2-algebras. Therefore, the equivalent relations between those structures are established. Furthermore, we study the coboundary strict pre-Lie 2-bialgebra, which leads to the 2-graded classical Yang-Baxter Equations on strict pre-Lie 2-algebras. We give an operator form description of the 2-graded classical Yang-Baxter Equations. Finally, we use the $\huaO$-operators on strict Lie 2-algebras to construct solutions of 2-graded classical Yang-Baxter Equations on strict pre-Lie 2-algebras.

In this paper, all the vector spaces are over algebraically closed field $\mathbb K$ of characteristic $0$ and finite dimensional.

\section{Preliminaries}\label{sec:Pre}
\subsection{Pre-Lie algebras and their cohomology}
\begin{defi}  A {\bf pre-Lie algebra} is a pair $(A,\cdot_A)$, where $A$ is a vector space and  $\cdot_A:A\otimes A\longrightarrow A$ is a bilinear multiplication
satisfying that for all $x,y,z\in A$, the associator
$(x,y,z)=(x\cdot_A y)\cdot_A z-x\cdot_A(y\cdot_A z)$ is symmetric in $x,y$,
i.e.
$$(x,y,z)=(y,x,z),\;\;{\rm or}\;\;{\rm
equivalently,}\;\;(x\cdot_A y)\cdot_A z-x\cdot_A(y\cdot_A z)=(y\cdot_A x)\cdot_A
z-y\cdot_A(x\cdot_A z).$$
\end{defi}

Let $(A,\cdot_A)$ be a pre-Lie algebra. The commutator $
[x,y]_A=x\cdot_A y-y\cdot_A x$ defines a Lie algebra structure
on $A$, which is called the {\bf sub-adjacent Lie algebra} of
$(A,\cdot_A)$ and denoted by $\huaG(A)$. Furthermore,
$L:A\longrightarrow \gl(A)$ with $x\mapsto L_x$, where
$L_xy=x\cdot_A y$, for all $x,y\in A$, gives a representation of
the Lie algebra $\huaG(A)$ on $A$. See \cite{Bai21} for more details.

\begin{defi}
	Let $(A,\cdot_A)$ be a pre-Lie algebra and $V$ a vector space. A {\bf representation} of $A$ on $V$ consists of a pair $(\rho, \mu)$, where $\rho:A \longrightarrow \gl(V)$ is a representation of the Lie algebra $\huaG(A)$ on $V$ and $\mu:A \longrightarrow \gl(V)$ is a linear map satisfying
\begin{eqnarray} 
	\label{representation of pre-Lie algebra}
	\rho(x)\mu(y)u-\mu(y)\rho(x)u=\mu(x\cdot_A y)u-\mu(y)\mu(x)u,, \quad \forall~x,y\in A, ~u\in V.
\end{eqnarray}
\end{defi}

Define $R:A\longrightarrow \gl(A)$ by $R_xy=y\cdot_A x$. Thus, $(A; L, R)$ is a representation of $(A,\cdot_A)$. Furthermore, $(A^*;\ad^*=L^*-R^*, -R^*)$ is also a representation of $(A,\cdot_A)$, where $L^*$ and $R^*$ are given by $$\left\langle L^*_x\xi, y\right\rangle=\left\langle\xi, -L_xy\right\rangle, \left\langle R^*_x\xi, y\right\rangle=\left\langle\xi, -R_xy\right\rangle, \quad \forall~x,y\in A, ~\xi\in A^*.$$

The cohomology complex for a pre-Lie algebra $(A,\cdot_A)$ with a representation $(V;\rho,\mu)$ is given as follows.
The set of $n$-cochains is given by
$C^n(A,V)=\Hom(\wedge^{n-1}A\otimes A,V),\
n\geq 1.$  For all $\phi\in C^n(A,V)$, the coboundary operator $\delta:C^n(A,V)\longrightarrow C^{n+1}(A,V)$ is given by
 \begin{eqnarray}\label{eq:pre-Lie cohomology}
 \nonumber\delta\phi(x_1, \cdots,x_{n+1})&=&\sum_{i=1}^{n}(-1)^{i+1}\rho({x_i})\phi(x_1, \cdots,\hat{x_i},\cdots,x_{n+1})\\
\label{eq:cobold} &&+\sum_{i=1}^{n}(-1)^{i+1}\mu({x_{n+1}})\phi(x_1, \cdots,\hat{x_i},\cdots,x_n,x_i)\\
 \nonumber&&-\sum_{i=1}^{n}(-1)^{i+1}\phi(x_1, \cdots,\hat{x_i},\cdots,x_n,x_i\cdot_A x_{n+1})\\
 \nonumber&&+\sum_{1\leq i<j\leq n}(-1)^{i+j}\phi([x_i,x_j]_A,x_1,\cdots,\hat{x_i},\cdots,\hat{x_j},\cdots,x_{n+1}),
\end{eqnarray}
for all $x_i\in A,~i=1,\cdots,n+1$.

\subsection{Lie 2-algebras and their cohomology}

\begin{defi}\label{Lie 2-algebra}{\rm (\cite{baez:2algebras})}
	A {\bf Lie $2$-algebra} is a $2$-term graded vector spaces $\g=\g_0\oplus\g_{-1}$ equipped with a linear $\mathfrak{d}:\g_{-1}\longrightarrow \g_0$, a skew-symmetric bilinear map $\mathfrak{l}_2:\g_i\times\g_j\longrightarrow\g_{i+j} ,-1\leq i+j \leq0$ and a skew-symmetric bilinear map $\mathfrak{l}_3:\land^3\g_0\longrightarrow\g_{-1}$, such that for any $x_i,x,y,z\in \g_0$ and $h,k\in \g_{-1}$, the following equalities are satisfied:
\begin{itemize}
\item[\rm(i)] $\mathfrak{d}\mathfrak{l}_2(x,h)=\mathfrak{l}_2(x,\mathfrak{d}h)$, \quad $\mathfrak{l}_2(\mathfrak{d}h,k)=\mathfrak{l}_2(h,\mathfrak{d}k)$,
\item[\rm(ii)] $\mathfrak{d}\mathfrak{l}_3(x,y,z)=\mathfrak{l}_2(x,\mathfrak{l}_2(y,z))+\mathfrak{l}_2(y,\mathfrak{l}_2(z,x))+\mathfrak{l}_2(z,\mathfrak{l}_2(x,y))$,
\item[\rm(iii)] $\mathfrak{l}_3(x,y,\mathfrak{d}h)=\mathfrak{l}_2(x,\mathfrak{l}_2(y,h))+\mathfrak{l}_2(y,\mathfrak{l}_2(h,x))+\mathfrak{l}_2(h,\mathfrak{l}_2(x,y))$,
\item[\rm(iv)]  the Jacobiator identity:$$\displaystyle\sum_{i=1}^{4}(-1)^{i+1}\mathfrak{l}_2(x_i,\mathfrak{l}_3(x_1,\cdots,\hat{x_i},\cdots, x_4))+\displaystyle\sum_{i<j}(-1)^{i+j}\mathfrak{l}_3(\mathfrak{l}_2(x_i,x_j),x_1,\cdots,\hat{x_i},\cdots,\hat{x_j},\cdots,x_4)=0.$$
\end{itemize}
 \end{defi}

Usually, we denote a Lie $2$-algebra by $(\g_0, \g_{-1}, \mathfrak{d}, \mathfrak{l}_2, \mathfrak{l}_3)$, or simply by $\g$. A Lie $2$-algebra is called {\bf strict} if $\mathfrak{l}_3=0$. We also denote a strict Lie $2$-algebra $(\g_0, \g_{-1}, \mathfrak{d}, [\cdot,\cdot]=\mathfrak{l}_2)$.

\begin{defi}\label{strict homomorphism}{\rm (\cite{baez:2algebras})}
	Let $\g=(\g_0, \g_{-1}, \mathfrak{d}, \mathfrak{l}_2)$ and ${\g}^{\prime}=({\g_0}^{\prime}, {\g_{-1}}^{\prime}, {\mathfrak{d}}^{\prime}, {\mathfrak{l}_2}^{\prime})$ be two strict Lie $2$-algebras. A {\bf strict homomorphism} $f$ from $\g$ to ${\g}^{\prime}$ consists of linear maps $f_0: \g_0 \longrightarrow {\g_0}^{\prime}$ and $f_1: \g_{-1} \longrightarrow {\g_{-1}}^{\prime}$, such that the following equalities hold for all $x, y, z \in \g_0, h \in \g_{-1}$,
\begin{itemize}
\item[\rm(i)]$f_0 \circ \mathfrak{d}={\mathfrak{d}}^{\prime} \circ f_1$,
\item[\rm(ii)]$f_0\mathfrak{l}_2(x, y)-{\mathfrak{l}_2}^{\prime}(f_0(x), f_0(y))=0$,
\item[\rm(iii)]$f_1\mathfrak{l}_2(x, h)-{\mathfrak{l}_2}^{\prime}(f_0(x), f_1(h))=0$.
\end{itemize}
\end{defi}

Let $\huaV:V_{-k+1} \stackrel{\partial}{\longrightarrow} V_{-k+2} \stackrel{\partial}{\longrightarrow} \cdots \stackrel{\partial}{\longrightarrow} V_{-1} \stackrel{\partial}{\longrightarrow} V_0$ be a complex of vector spaces. Define ${\rm End}_\partial^0(\huaV)$ by
 $${\rm End}_\partial^0(\huaV)=\{E \in \displaystyle\oplus_{i=0}^{k-1} {\rm End}(V_i)~ |~ E \circ \partial = \partial \circ E\},$$
  and define $${\rm End}^{-1}(\huaV)=\{E \in \displaystyle\oplus_{i=0}^{k-2} {\rm Hom}(V_i, V_{i+1})~ |~ [E, E]_C=0\},$$ where $[\cdot, \cdot]_C$ is the natural commutator. Then there is a differential $\delta:{\rm End}^{-1}(\huaV) \longrightarrow {\rm End}_d^0(\huaV)$ given by $$\delta(\phi)=\phi \circ \partial + \partial \circ \phi, \quad \forall ~\phi \in {\rm End}^{-1}(\huaV).$$
Then ${\rm End}^{-1}(\huaV)\stackrel{\delta}{\longrightarrow}{\rm End}_\partial^0(\huaV)$ is a strict Lie $2$-algebra, which we denote by ${\rm End}(\huaV)$.

\begin{defi}{\rm (\cite{shengzhu2})}
A {\bf strict representation} of a strict Lie $2$-algebra $\g$ on a $k$-term complex of vector spaces $\huaV$ is a strict homomorphism $\rho=(\rho_0, \rho_1)$ from $\g$ to the strict Lie $2$-algebra ${\rm End}(\huaV)$. We denote a strict representation by $(\huaV; \rho)$.
\end{defi}

Let $\huaV:V_{-1} \stackrel{\partial}{\longrightarrow} V_0$ be a strict representation of the strict Lie $2$-algebra $\g$. 
Define a strict Lie $2$-algebra structure on $\g\oplus \huaV$, in which the degree 0 part is $\g_0 \oplus V_0$, the degree 1 part is $\g_{-1} \oplus V_{-1}$, the differential is $\mathfrak{d}+\partial:\g_{-1} \oplus V_{-1} \longrightarrow \g_0 \oplus V_0$, and for all $x,y\in \g_0, h\in \g_{-1}, u,v\in V_0, m\in V_{-1}$, $\mathfrak{l}_2^s$ is given by
\begin{eqnarray}
\mathfrak{l}_2^s(x+u, y+v) &=& \mathfrak{l}_2(x, y)+\rho_0(x)v-\rho_0(y)u,\\
\mathfrak{l}_2^s(x+u, h+m) &=& \mathfrak{l}_2(x, h)+\rho_0(x)m-\rho_1(h)u.
\end{eqnarray}
Then there is a semidirect product strict Lie $2$-algebra $\g\ltimes\huaV$.

For a strict representation $(\huaV; \rho)$ of a strict Lie $2$-algebra $\g$, and let $\huaV^*:V_0^*\stackrel{\partial^*}{\longrightarrow}V_{-1}^*\stackrel{\partial^*}{\longrightarrow}\cdots \stackrel{\partial^*}{\longrightarrow}V_{-k+1}^*$ be the dual complex of $\huaV$. Define $\rho_0^*:\g_0\longrightarrow {\rm End}_{\partial^*}^0(\huaV^*)$ and $\rho_1^*:\g_{-1}\longrightarrow {\rm End}^{-1}(\huaV^*)$ by
\begin{eqnarray*}
\left\langle\rho_0^*(x)u^*, v\right\rangle&=&-\left\langle u^*, \rho_0(x)v\right\rangle, \quad \forall~u^*\in V_i^*, v\in V_i,\\
\left\langle\rho_1^*(h)u^*, v\right\rangle&=&-\left\langle u^*, \rho_1(h)v\right\rangle, \quad \forall~u^*\in V_i^*, v\in V_{i+1}.
\end{eqnarray*}
Then $\rho^*=(\rho_0^*, \rho_1^*)$ is a strict representation of the strict Lie $2$-algebra of $\g$ on $\huaV^*$, which is called the {\bf dual representation} of the representation $(\huaV;\rho)$.

For two complexes of vector spaces
$\huaV:V_{-k+1}\stackrel{\partial_V}{\longrightarrow}\cdots\stackrel{\partial_V}{\longrightarrow}V_0$
and
$\huaW:W_{-m+1}\stackrel{\partial_W}{\longrightarrow}\cdots\stackrel{\partial_W}{\longrightarrow}W_0$,
their tensor product $\huaV\otimes\huaW$ is also a complex of vector
spaces. The degree $-n$ part $(\huaV\otimes\huaW)_{-n}$ is given by
$$
(\huaV\otimes\huaW)_{-n}=\sum_{i+j=-n}V_i\otimes W_j,
$$
and $\partial$ is the tensor product of $\partial_V$ and
$\partial_W$, i.e.,
\begin{equation}\label{eq:tensor}
\partial(v\otimes w)=(\partial_V\otimes1+(-1)^{|v|+1}1\otimes\partial_W)(v\otimes w)=\partial_Vv\otimes
w+(-1)^{|v|+1}v\otimes\partial_Ww,
\end{equation}
for any $v\in\huaV$ and $w\in\huaW$. Furthermore, let $(\huaV; \rho^V)$ and $(\huaV; \rho^W)$ be two strict representations of $\g$. Then the tensor product $(\huaV\otimes\huaW; \rho)$ is also a strict representation of $\g$, where $\rho=(\rho_0, \rho_1)$ is given by $$\rho_0=\rho_0^V\otimes 1+1\otimes \rho_0^W, \quad \rho_1=\rho_1^V\otimes 1+1\otimes \rho_1^W.$$

The {\bf adjoint representation} of $\g$ on itself is denoted by $\ad=(\ad_0, \ad_1)$, with 
$$\ad_0(x)=[x, \cdot] \in \End_\partial^0(\g), \quad \ad_1(h)=[h, \cdot] \in \End^{-1}(\g),$$
 which is a strict representation. The dual representation of $\g$ on $\g^*$ is called the {\bf coadjoint representation} and denoted by $\ad^*=(\ad_0^*, \ad_1^*)$.

The cohomology complex for a strict Lie $2$-algebra $\g$ with a $k$-term strict representation $(\huaV; \rho)$ is given as follows. 

Given a  strict representation $(\huaV; \rho)$ of a strict Lie $2$-algebra $\g$, we have the corresponding generalized Chevalley-Eilenberg complex $(C^n(\g,\huaV),D)$, where the $n$-cochains $C^n(\g,\huaV)$ is defined by
$$C^{n}(\g,\huaV)=\oplus_{n+1} (\odot^k\g^*[-1]) \otimes V_{n+1-k},$$
and the coboundary operator $D$ is  defined by
 $$D=\hat{\mathfrak{\frkd}}+\bar{\mathfrak{d}}+\hat{\partial}:C^n(\g, \huaV)\longrightarrow C^{n+1}(\g, \huaV),$$
where the operator $\hat{\frkd}: \Hom((\land^p\g_0)\land(\odot^q\g_{-1}), V_s)\longrightarrow \Hom((\land^{p-1}\g_0)\land(\odot^{q+1}\g_{-1}), V_s)$ is of degree $1$ is defined by
 $$\hat{\frkd}(f)(x_1, \cdots, x_{p-1}, h_1, h_2, \cdots, h_{q+1})=(-1)^p(f(x_1, \cdots, x_{p-1}, \frkd h_1, h_2, \cdots, h_{q+1})+{\rm c.p.} (h_1, \cdots, h_{q+1})),$$ where $f \in C^p(\g, \huaV)$, $x \in \g_0$ and $h_j \in \g_{-1}$, the operator $\hat{\partial}: \Hom((\land^p\g_0)\land(\odot^q\g_{-1}), V_s)\longrightarrow \Hom((\land^{p}\g_0)\land(\odot^q\g_{-1}), V_{s+1})$ is of degree $1$ is defined by 
  $$\hat{\partial}(f)=(-1)^{p+2q}\partial\circ f,$$ 
 and the operator $\bar{\frkd}=(\bar{\frkd}^{(1,0)}, \bar{\frkd}^{(0,1)})$, where $\bar{\frkd}^{(1,0)}:\Hom((\land^p\g_0)\land(\odot^q\g_{-1}), V_s)\longrightarrow \Hom((\land^{p+1}\g_0)\land(\odot^{q}\g_{-1}), V_s)$ is defined by 
\begin{eqnarray*}
&&\bar{\frkd}^{(1,0)}(f)(x_1, \cdots, x_{p+1}, h_1, \cdots, h_q)\\
&=&\sum_{i=1}^{p+1}(-1)^{i+1}\rho_0(x_i)f(x_1, \cdots, \hat{x_i}, \cdots, x_{p+1}, h_1, \cdots, h_q)\\
&&+\displaystyle\sum_{1\leq i<j\leq p+1}(-1)^{i+j}f([x_i, x_j], x_1, \cdots, \hat{x_i}, \cdots, \hat{x_j}, \cdots, x_{p+1}, h_1, \cdots, h_q)\\
&&+\displaystyle\sum_{1\leq i\leq p+1, 1\leq j\leq q}(-1)^{i}f(x_1, \cdots, \hat{x_i}, \cdots, x_{p+1}, h_1, \cdots, [x_i, h_j], \cdots, h_q),
\end{eqnarray*}
and $\bar{\frkd}^{(0,1)}:\Hom((\land^p\g_0)\land(\odot^q\g_{-1}), V_s)\longrightarrow \Hom((\land^{p}\g_0)\land(\odot^{q+1}\g_{-1}), V_{s-1})$ is defined by
\begin{eqnarray*}
&&\bar{\frkd}^{(0,1)}(f)(x_1, \cdots, x_p, h_1, \cdots, h_{q+1})\\
&=&\sum_{i=1}^{q+1}(-1)^{p}\rho_1(h_i)f(x_1, \cdots, x_p, h_1, \cdots, \hat{h_i}, \cdots, h_{q+1}).
\end{eqnarray*}
The generalized
Chevalley-Eilenberg complex can be explicitly given by
\begin{eqnarray*}
 && V_{-k+1}\stackrel{D}{\longrightarrow} V_{-k+2}\oplus
  \Hom(\frkg_{0},V_{-k+1})\stackrel{D}{\longrightarrow}\\
&&  V_{-k+3} \oplus\Hom(\frkg_{0},V_{-k+2})\oplus
  \Hom(\frkg_{-1},V_{-k+1})\oplus\Hom(\wedge  ^2\frkg_0,V_{-k+1})\stackrel{D}{\longrightarrow}\\
 &&V_{-k+4}\oplus
  \Hom(\wedge  ^2\frkg_0,V_{-k+2})\oplus\Hom(\frkg_{-1},V_{-k+2})\oplus\Hom(\wedge  ^3\frkg_0,V_{-k+1})\oplus\Hom(\frkg_0\wedge \frkg_{-1},V_{-k+1})\\
 && \stackrel{D}{\longrightarrow}\cdots.
\end{eqnarray*}
See \cite{BSZ,shengzhu1,shengzhu2} for more details on representations and cohomology of Lie 2-algebras.

\section{Symplectic strict Lie 2-algebras, strict pre-Lie 2-algebras and their representations}\label{sec:Pre-Lie 2-algebras}
In this section, we first recall the basic notions and properties of strict pre-Lie 2-algebras. Then we give the notion of a symplectic structure on a strict Lie 2-algebra and show that a symplectic strict Lie 2-algebra can give rise to a strict pre-Lie 2-algebra. We also give the constructions of strict pre-Lie 2-algebras through $\huaO$-operators on strict Lie 2-algebras, Rota-Baxter operators of weight $\lambda$ on strict associative 2-algebras and derivations on strict commutative associative 2-algebras. Finally, we give the representations of strict pre-Lie 2-algebras, which can give the semidirect product constructions of strict pre-Lie $2$-algebras.

\subsection{Symplectic strict Lie 2-algebras and constructions of strict pre-Lie 2-algebras}

\begin{defi}{\rm (\cite{Sheng19})}\label{defi:2-term pre}
	A {\bf pre-Lie $2$-algebra} is a $2$-term graded vector spaces $\huaA=A_0\oplus A_{-1}$, together with linear maps $\dM:A_{-1}\longrightarrow A_0$, $\cdot:A_i\otimes A_j\longrightarrow A_{i+j}$, $-1\leq i+j\leq 0$, and $l_3:\wedge^2A_0\otimes A_0\longrightarrow A_{-1}$, such that for all $x,x_i\in A_0$ and $a,b\in A_{-1}$, we have
	\begin{itemize}
		\item[\rm($a_1$)] $\dM (x\cdot a)=x\cdot \dM a,$
		\item[\rm($a_2$)] $ \dM(a\cdot x)=(\dM a)\cdot x,$
		\item[\rm($a_3$)]$ \dM a\cdot b=a\cdot \dM b,$
		\item[\rm($b_1$)]$x_0\cdot(x_1\cdot x_2)-(x_0\cdot x_1)\cdot x_2-x_1\cdot(x_0\cdot x_2)+(x_1\cdot x_0)\cdot x_2=\dM l_3(x_0,x_1,x_2),$
		
		\item[\rm($b_2$)]$x_0\cdot(x_1\cdot a)-(x_0\cdot x_1)\cdot a-x_1\cdot(x_0\cdot a)+(x_1\cdot x_0)\cdot a=l_3(x_0,x_1,\dM a),$
		
		\item[\rm($b_3$)]$a\cdot(x_1\cdot x_2)-(a\cdot x_1)\cdot x_2-x_1\cdot(a\cdot x_2)+(x_1\cdot a)\cdot x_2=l_3(\dM  a,x_1,x_2),$
		
		\item[\rm(c)]\begin{eqnarray*}
			&&x_0\cdot l_3(x_1,x_2,x_3)- x_1\cdot l_3(x_0,x_2,x_3)+ x_2\cdot l_3(x_0,x_1,x_3)\\
			&&+l_3(x_1,x_2,x_0)\cdot x_3- l_3(x_0,x_2,x_1)\cdot x_3+ l_3(x_0,x_1,x_2)\cdot x_3\\&&-l_3( x_1,x_2,x_0\cdot x_3)+l_3( x_0,x_2,x_1\cdot x_3)-l_3(x_0,x_1, x_2\cdot x_3)\\
			&&-l_3(x_0\cdot x_1-x_1\cdot  x_0, x_2,x_3)+l_3(x_0\cdot x_2-x_2\cdot x_0,x_1,x_3)-l_3(x_1\cdot x_2-x_2\cdot x_1,x_0,x_3)=0.
		\end{eqnarray*}
	\end{itemize}
\end{defi}
 Usually, we denote a pre-Lie $2$-algebra by $(A_0, A_{-1}, \dM, \cdot, l_3)$, or simply by $\huaA$. A pre-Lie $2$-algebra $(A_0, A_{-1}, \dM, \cdot, l_3)$ is said to be {\bf strict} if $l_3=0$.
 
 Given a pre-Lie $2$-algebra  $(A_0,A_{-1},\dM,\cdot,l_3)$, we
 define $\frkl_2:A_i\wedge A_j\longrightarrow A_{i+j}$ and $\frkl_3:\wedge^3A_0\longrightarrow A_{-1}$ by
 \begin{eqnarray}
 	\label{eq:l21}\frkl_2(x,y)&=&x\cdot y-y\cdot x,\\
 	\label{eq:l22} \frkl_2(x,a)&=&-\frkl_2(a,x)=x\cdot a-a\cdot x,\\
 	\label{eq:l3}\frkl_3(x,y,z)&=&l_3(x,y,z)+l_3(y,z,x)+l_3(z,x,y).
 \end{eqnarray}
 Furthermore, define $L_0:A_0\longrightarrow\End(A_0)\oplus \End(A_{-1})$ by
 \begin{equation}\label{eq:L0}
 	L_0(x)y=x\cdot y,\quad L_0(x)a=x\cdot a.
 \end{equation}
 Define $L_1:A_{-1}\longrightarrow\Hom(A_0,A_{-1})$ by
 \begin{equation}\label{eq:L1}
 	L_1(a)x=a\cdot x.
 \end{equation}
 Define $L_2:\wedge^2A_0\longrightarrow \Hom(A_0,A_{-1})$ by
 \begin{equation}\label{eq:L2}
 	L_2(x,y)z=-l_3(x,y,z),\quad \forall x,y,z\in A_0.
 \end{equation}
 \begin{thm}\label{thm:main1}{\rm (\cite{Sheng19})}
 	Let $\huaA=(A_0,A_{-1},\dM,\cdot,l_3)$ be a pre-Lie $2$-algebra. Then $(A_0,A_{-1},\dM,\frkl_2,\frkl_3)$ is a Lie $2$-algebra, which we denote by $\huaG(\huaA)$, where $\frkl_2$ and $\frkl_3$ are given by \eqref{eq:l21}-\eqref{eq:l3} respectively. Furthermore, $(L_0,L_1,L_2)$ is a representation of the Lie $2$-algebra $\huaG(\huaA)$ on the complex of vector spaces $A_{-1}\stackrel{\dM}{\longrightarrow} A_0$, where $L_0,L_1,L_2$ are given by \eqref{eq:L0}-\eqref{eq:L2} respectively.
 \end{thm}

\begin{cor}
Let $\huaA=(A_0,A_{-1},\dM,\cdot)$ be a strict pre-Lie $2$-algebra. Then $(A_0,A_{-1},\dM,\frkl_2)$ is a strict Lie $2$-algebra, which we also denote by $\huaG(\huaA)$, where $\frkl_2$ is given by \eqref{eq:l21} and \eqref{eq:l22}. Furthermore, $(L_0,L_1)$ is a strict representation of the strict Lie $2$-algebra $\huaG(\huaA)$ on the complex of vector spaces $A_{-1}\stackrel{\dM}{\longrightarrow} A_0$, where $L_0$ and $L_1$ are given by \eqref{eq:L0} and \eqref{eq:L1} respectively.

\end{cor}

\begin{pro}\label{pro:semi-product}
Let $\huaA=(A_0,A_{-1},\dM,\cdot)$ be a strict pre-Lie $2$-algebra. Define $\rho:A_0\rightarrow \gl(A_{-1})$ and $\mu:A_0\rightarrow \gl(A_{-1})$ by
$$\rho(x)(a)=x\cdot a,\quad \mu(x)(a)=a\cdot x,\quad x\in A_0,a\in A_{-1}.$$
Then $(A_{-1};\rho,\mu)$ is a representation of the pre-Lie algebra $A_0$ and thus we obtain a semi-product pre-Lie algebra $A_{0}\ltimes_{(\rho, \mu)}  A_{-1}$ with the pre-Lie operation defined by
\begin{equation}\label{eq:semi-product}
  (x+a)\ast(y+b)=x\cdot y+\rho(x)b+\mu(y)a,\quad x,y\in A_0,a,b\in A_{-1}.
\end{equation}
\end{pro}

\begin{defi}
 Let    $\huaA=(A_0,A_1,\dM,\cdot,l_3)$ and $\huaA'=(A_0',A_1',\dM',\cdot',l_3')$ be pre-Lie $2$-algebras. A {\bf homomorphism} $(F_0,F_1,F_2)$ from $\huaA$ to $\huaA'$ consists of linear maps $F_0:A_0\longrightarrow A_0'$,  $F_1:A_1\longrightarrow A_1'$, and $F_2:A_0\otimes A_0\longrightarrow A_1'$ such that the following equalities hold:
 \begin{itemize}
  \item[\rm(i)] $F_0\circ \dM=\dM'\circ F_1,$
   \item[\rm(ii)] $F_0(u\cdot v)-F_0(u)\cdot'F_0( v)=\dM'F_2(u,v),$
    \item[\rm(iii)] $F_1(u\cdot m)-F_0(u)\cdot'F_1( m)=F_2(u,\dM m),\quad F_1( m\cdot u)-F_1( m)\cdot'F_0(u)=F_2(\dM m,u),$
     \item[\rm(iv)] $F_0(u)\cdot'F_2(v,w)-F_0(v)\cdot'F_2(u,w)+F_2(v,u)\cdot'F_0(w)-F_2(u,v)\cdot'F_0(w)-F_2(v,u\cdot w)\\+F_2(u,v\cdot w)-F_2(u\cdot v ,w)+F_2(v\cdot u ,w)+l_3'(F_0(u),F_0(v),F_0(w))-F_1l_3(u,v,w)=0$.
 \end{itemize}
\end{defi}

 For a strict Lie $2$-algebra $\g$ with a trivial representation $\rho=0$ on the complex $\K_{-1}\stackrel{0}{\longrightarrow} \K_0$ with $\K_{-1}=\K_0=\K$, now the generalized
Chevalley-Eilenberg complex can be explicitly given by
\begin{eqnarray*}
 &&\Hom(\frkg_{0},\K_{-1})\stackrel{\huaD}{\longrightarrow} \Hom(\frkg_{0},\K_0)\oplus \Hom(\frkg_{-1},\K_{-1})\oplus\Hom(\wedge^2\frkg_0,\K_{-1})
  \stackrel{\huaD}{\longrightarrow}\\
&&  \Hom(\wedge^2\frkg_0,\K_0)\oplus\Hom(\frkg_0\wedge \g_{-1},\K_{-1})\oplus\Hom(\frkg_{-1},\K_0)\oplus
  \Hom(\wedge^3 \g_0,\K_{-1})\stackrel{\huaD}{\longrightarrow}\\
 &&\Hom(\wedge^3\frkg_0,\K_0)\oplus\Hom(\wedge^{2}\frkg_{-1},\K_{-1})\oplus
  \Hom(\wedge^2\frkg_0\wedge \g_{-1},\K_{-1})\oplus\Hom(\frkg_0\wedge \g_{-1},\K_{0})\\
  &&\oplus \Hom(\wedge^4\frkg_0,\K_{-1})\stackrel{\huaD}{\longrightarrow} \cdots.
\end{eqnarray*}

A $2$-cochain $\omega=(\omega_1,\omega_2)\in  \Hom(\wedge^2\frkg_0,\K_0)\oplus\Hom(\frkg_0\wedge \g_{-1},\K_{-1})$ is closed, that is, $\huaD \omega=0$ if and only if for $x,y,z\in\g_0$, $h,k\in\g_{-1}$, the following equalities hold:
\begin{eqnarray}
\label{eq:closed1}\omega_1([x,y],z)+\omega_1([y,z],x)+\omega_1([z,x],y)&=&0,\\
\label{eq:closed2}\omega_2([x,y],h)+\omega_2([h,x],y)+\omega_2([y,h],x)&=&0\\
\label{eq:closed3}\omega_1(x,\frkd h)=0,\quad \omega_2(h,\frkd k)=\omega_2(\frkd h,k).&& 
\end{eqnarray}

Recall a graded skew-symmetric bilinear form $\omega$ on a 2-term vector spaces $\g=\g_0\oplus \g_{-1}$ is a bilinear map $\omega:\g_i\otimes \g_j\rightarrow \K,~-1\leq i,j\leq 0$ satisfying
 \begin{equation}
\omega(u,v)=-(-1)^{|u||v|}\omega(v,u),\quad u,v\in \g.
 \end{equation}

 \begin{defi}
Let $\g$ be a  strict Lie $2$-algebra. A pair $(\omega_1,\omega_2)\in  \Hom(\wedge^2\frkg_0,\K_0)\oplus\Hom(\frkg_0\wedge \g_{-1},\K_{-1})$ is called a {\bf symplectic structure} on $\g$ if $(\omega_1,\omega_2)$ is a  graded skew-symmetric nondegenerate closed $2$-form. Furthermore, a  strict Lie $2$-algebra $\g$ with a symplectic structure  $(\omega_1,\omega_2)$ is called a {\bf symplectic strict Lie $2$-algebra}. We denote a symplectic strict Lie $2$-algebra by $(\g;(\omega_1,\omega_2))$.
\end{defi}

\begin{pro}\label{pro:symplectic strict L2-preL2}
Let $(\g;(\omega_1,\omega_2))$ be a symplectic strict Lie $2$-algebra. Define bilinear operations $\cdot:\g_0\otimes \g_0\rightarrow \g_0$, $\cdot:\g_0\otimes \g_{-1}\rightarrow\g_{-1}$ and $\cdot:\g_{-1}\otimes \g_0\rightarrow\g_{-1}$ by
\begin{eqnarray}
 \label{eq:sym-preL21} \omega_1(x\cdot y,z)&=&-\omega_1(y,[x,z]),\\
 \label{eq:sym-preL22} \omega_2(x\cdot a,y)&=&-\omega_2(a,[x,y]),\\
 \label{eq:sym-preL23} \omega_2(a\cdot x,y)&=&-\omega_2(y\cdot x,a)=-\omega_2(x,[a,y]),\quad x,y,z\in\g_0,a\in\g_{-1}.
\end{eqnarray}
Then $(\g_0,\g_{-1},\frkd,\cdot)$ is a strict pre-Lie $2$-algebra.
\end{pro}
\begin{proof}
	For any $x\in \g_0$, $a,b \in \g_{-1}$, by Condition (i) of the Definition \ref{Lie 2-algebra}, \eqref{eq:closed3}, \eqref{eq:sym-preL22} and \eqref{eq:sym-preL23}, we have
    \begin{eqnarray*}
    	\omega_2(\frkd(x\cdot a), b)&=&\omega_2(x\cdot a, \frkd b)=-\omega_2(a, [x, \frkd b])\\
        &=&-\omega_2(a, \frkd[x, b]))=-\omega_2(\frkd a, [x, b])=\omega_2(x\cdot \frkd a, b),
    \end{eqnarray*}
  which implies that $\frkd(x\cdot a)=x\cdot \frkd a$.
  
    Similarly, we can get $$\frkd(a\cdot x)=(\frkd a)\cdot x, \quad \frkd a\cdot b=a\cdot \frkd b.$$
    
    By \eqref{eq:sym-preL21}, we have
    \begin{eqnarray*}
    	\omega_1(x_0\cdot(x_1\cdot x_2), z)&=&-\omega_1(x_1\cdot x_2, [x_0, z])~=~\omega_1(x_2, [x_1, [x_0, z]]),\\
    	\omega_1((x_0\cdot x_1)\cdot x_2, z)&=&-\omega_1(x_2, [x_0\cdot x_1, z]),\\
    	\omega_1(x_1\cdot(x_0\cdot x_2), z)&=&-\omega_1(x_0\cdot x_2, [x_1, z])~=~\omega_1(x_2, [x_0, [x_1, z]]),\\
    	\omega_1((x_1\cdot x_0)\cdot x_2, z)&=&-\omega_1(x_2, [x_1\cdot x_0, z]).
    \end{eqnarray*}
 By Condition \rm(ii) of the Definition \ref{Lie 2-algebra} and non-degeneracy of $\omega_1$, we have 
    \begin{eqnarray*}
    	&&\omega_1(x_0\cdot(x_1\cdot x_2)- (x_0\cdot x_1)\cdot x_2-x_1\cdot(x_0\cdot x_2)+(x_1\cdot x_0)\cdot x_2, z)\\
    	&=&\omega_1(x_2, [x_1, [x_0, z]]+[x_0\cdot x_1, z]-[x_0, [x_1, z]]-[x_1\cdot x_0, z])\\
    	&=&\omega_1(x_2, [x_1, [x_0, z]]+[z, [x_1, x_0]]+[x_0, [z, x_1]])=0,
    \end{eqnarray*}
   which implies that
    \begin{eqnarray*}
    	x_0\cdot(x_1\cdot x_2)- (x_0\cdot x_1)\cdot x_2-x_1\cdot(x_0\cdot x_2)+(x_1\cdot x_0)\cdot x_2=0
    \end{eqnarray*}
    holds.
    
    Similarly, by 
    Conditions {\rm(ii)} and {\rm(iii)} of the Definition \ref{Lie 2-algebra}, we can get
    \begin{eqnarray*}
    	&&\omega_2(x_0\cdot(x_1\cdot a)-(x_0\cdot x_1)\cdot a-x_1\cdot(x_0\cdot a)+(x_1\cdot x_0)\cdot a, z)\\
    	&=&\omega_2(a, [x_1, [x_0, z]]+[x_0\cdot x_1, z]-[x_0, [x_1, z]]-[x_1\cdot x_0, z])\\
    	&=&\omega_2(a, [x_0, [z, x_1]]+[x_1, [x_0, z]]+[z, [x_1, x_0]])=0;\\
    	&&\omega_2(a\cdot(x_1\cdot x_2)-(a\cdot x_1)\cdot x_2-x_1\cdot(a\cdot x_2)+(x_1\cdot a)\cdot x_2, z)\\
    	&=&\omega_2(x_2, [x_1, [a, z]]+[a\cdot x_1, z]-[a, [x_1, z]]-[x_1\cdot a, z])\\
    	&=&\omega_2(x_2, [x_1, [a, z]]+[a, [z, x_1]]+[z, [x_1, a]])=0.
    \end{eqnarray*}
    Then by the non-degeneracy of $\omega_2$, we have
    \begin{eqnarray*}
    	x_0\cdot(x_1\cdot a)-(x_0\cdot x_1)\cdot a-x_1\cdot(x_0\cdot a)+(x_1\cdot x_0)\cdot a&=&0;\\
    	a\cdot(x_1\cdot x_2)-(a\cdot x_1)\cdot x_2-x_1\cdot(a\cdot x_2)+(x_1\cdot a)\cdot x_2&=&0.
    \end{eqnarray*}
    Therefore, $(\g_0,\g_{-1},\frkd,\cdot)$ is a strict pre-Lie $2$-algebra.
\end{proof}

Let $\g=(\g_0,\g_{-1}, \frkd, \frkl_2)$ be a strict Lie $2$-algebra and $(\rho_0,\rho_1)$ be a strict representation of $\g$ on a $2$-term complex of vector spaces $\huaV=V_{-1}\stackrel{\dM}{\longrightarrow}V_0$.
\begin{defi}\label{defi:O-operator}{\rm (\cite{Sheng19})}
  A pair $(T_0,T_1)$, where $T_0:V_0\longrightarrow \g_0$, $T_1:V_{-1}\longrightarrow \g_{-1}$ is a chain map, is called an  {\bf $\mathcal{O}$-operator} on $\g$ associated to the representation  $(\rho_0,\rho_1)$, if for all $u,v,v_i\in V_0$ and $m\in V_{-1}$ the following conditions are satisfied:
  \begin{itemize}
    \item[\rm(i)] $T_0\big(\rho_0(T_0u)v-\rho_0(T_0v)u\big)-\frkl_2(T_0u,T_0v)=0;$

    \item[\rm(ii)] $T_1\big(\rho_1(T_1m)v-\rho_0(T_0v)m\big)-\frkl_2(T_1m,T_0v)= 0$.
    \end{itemize}
    In particular, the $\mathcal{O}$-operator $(T_0,T_1)$  associated to the representation $(\g;\ad_0,\ad_1)$ is call a {\bf Rota-Baxter operator} on $\g$.
    \end{defi}

    \begin{pro}{\rm (\cite{Sheng19})}\label{pro:RB-Lie2}
      Let $(\rho_0,\rho_1)$ be a strict representation of $\g$ on $\huaV$ and $(T_0,T_1)$ an  $\mathcal{O}$-operator on $\g$ associated to the representation  $(\rho_0,\rho_1)$. Define a degree $0$ multiplication $\cdot:V_i\otimes V_j\longrightarrow V_{i+j}$, $-1\leq i+j\leq 0$, on $\huaV$ by
    \begin{equation}\label{eq:formularm}
      u\cdot v=\rho_0(T_0u)v,\quad u\cdot m=\rho_0(T_0u)m,\quad m\cdot u=\rho_1(T_1m)u.
    \end{equation}  
      Then, $(V_0,V_{-1},\dM,\cdot)$ is a strict pre-Lie  $2$-algebra.
    \end{pro}

\begin{defi}\label{defi:2-term ass}{\rm (\cite{Khm})}
	An {\bf associative  $2$-algebra} is a $2$-term graded vector spaces $\huaA=A_0\oplus A_{-1}$, together with linear maps $\dM:A_{-1}\longrightarrow A_0$, $\cdot:A_i\otimes A_j\longrightarrow A_{i+j}$, $-1\leq i+j\leq 0$, and $l_3:\wedge^2A_0\otimes A_0\longrightarrow A_{-1}$, such that for all $x,x_i\in A_0$ and $a,b\in A_{-1}$, we have
	\begin{itemize}
		\item[\rm($a_1$)] $\dM (x\cdot a)=x\cdot \dM a,$
		\item[\rm($a_2$)] $ \dM(a\cdot x)=(\dM a)\cdot x,$
		\item[\rm($a_3$)]$ \dM a\cdot b=a\cdot \dM b,$
		\item[\rm($b_1$)]$x_0\cdot(x_1\cdot x_2)-(x_0\cdot x_1)\cdot x_2=\dM l_3(x_0,x_1,x_2),$
		
		\item[\rm($b_2$)]$x_0\cdot(x_1\cdot a)-(x_0\cdot x_1)\cdot a=l_3(x_0,x_1,\dM a),$
        
    \item[\rm($b_3$)]$x_0\cdot(a\cdot x_1)-(x_0\cdot a)\cdot x_1=l_3(x_0,\dM a,x_1),$
        
		\item[\rm($b_4$)]$a\cdot(x_1\cdot x_2)-(a\cdot x_1)\cdot x_2=l_3(\dM  a,x_1,x_2),$
		
		\item[\rm(c)]\begin{eqnarray*}
			&&x_0\cdot l_3(x_1,x_2,x_3)+l_3(x_1,x_2,x_0)\cdot x_3\\&&-l_3( x_0,x_1,x_2\cdot x_3)+l_3( x_0,x_1\cdot x_2, x_3)-l_3(x_0\cdot x_1,x_2, x_3)=0.
		\end{eqnarray*}
	\end{itemize}
\end{defi}
Usually, we denote an associative $2$-algebra by $(A_0, A_{-1}, \dM, \cdot, l_3)$, or simply by $\huaA$. An associative $2$-algebra $(A_0, A_{-1}, \dM, \cdot, l_3)$ is said to be {\bf strict} if $l_3=0$.

\begin{defi}\label{a Rota-Baxter operator}
Let $\huaA$ be a strict associative $2$-algebra. A pair $(\huaR_0,\huaR_1)$, where $\huaR_0:A_0\longrightarrow A_0,\huaR_1:A_{-1}\longrightarrow A_{-1}$ is a chain map, is called a  {\bf Rota-Baxter operator of weight $\lambda$} on $\huaA$  if for all $x,y\in A_0$ and $a\in A_{-1}$, the following conditions are satisfied:
  \begin{itemize}
    \item[\rm(i)] $\huaR_0\big(\huaR_0(x)\cdot y+x\cdot \huaR_0(y)-\lambda (x\cdot y)\big)-\huaR_0(x)\cdot \huaR_0(y)=0;$

    \item[\rm(ii)] $\huaR_1\big(\huaR_0(x)\cdot a+x\cdot \huaR_1(a)-\lambda(x\cdot a)\big)-\huaR_0(x)\cdot \huaR_1(a)=0;$
    
    \item[\rm(iii)] $\huaR_1\big(\huaR_1(a)\cdot x+a\cdot \huaR_0(x)-\lambda(a\cdot x)\big)-\huaR_1(a)\cdot \huaR_0(x)=0$.
    \end{itemize}
\end{defi}

\begin{pro}
Let $(\huaR_0,\huaR_1)$ be a Rota-Baxter operator of weight $0$ on a strict associative $2$-algebra $\huaA$. 
Define a degree $0$ multiplication $\cdot_\huaR:A_i\otimes A_j\longrightarrow A_{i+j}$, $-1\leq i+j\leq 0$, on $\huaA$ by
    \begin{eqnarray}
      x\cdot_\huaR y&=&\huaR_0(x)\cdot y-y\cdot \huaR_0(x),\\
      x\cdot_\huaR a&=&\huaR_0(x)\cdot a-a\cdot \huaR_0(x),\\
      a\cdot_\huaR x&=&\huaR_1(a)\cdot x-x\cdot \huaR_1(a),\quad x,y\in A_0,a\in A_{-1}.
    \end{eqnarray}
    Then, $(A_0,A_{-1},\dM,\cdot_\huaR)$ is a strict pre-Lie  $2$-algebra.
\end{pro}
\begin{proof}
Let $(\huaR_0,\huaR_1)$ be a Rota-Baxter operator of weight $0$ on a strict associative $2$-algebra $\huaA$. Then  $(\huaR_0,\huaR_1)$ is a Rota-Baxter operator on the strict Lie $2$-algebra $(\huaA_0,\huaA_{-1},\dM,\frkl_2)$, where $\frkl_2:A_i\otimes A_j\longrightarrow A_{i+j}$ is given by
\begin{eqnarray}
 	\label{eq:R1}\frkl_2(x,y)&=&x\cdot y-y\cdot x,\\
 	\label{eq:R2} \frkl_2(x,a)&=&-\frkl_2(a,x)=x\cdot a-a\cdot x.
 \end{eqnarray}
 By Proposition \ref{pro:RB-Lie2}, $(A_0,A_{-1},\dM,\cdot_\huaR)$ is a strict pre-Lie  $2$-algebra.
\end{proof}

\begin{pro}
Let $(\huaR_0,\huaR_1)$ be a Rota-Baxter operator of weight $1$ on a strict associative $2$-algebra $\huaA$. 
Define a degree $0$ multiplication $\cdot_\huaR:A_i\otimes A_j\longrightarrow A_{i+j}$, $-1\leq i+j\leq 0$, on $\huaA$ by
    \begin{eqnarray}
      x\cdot_\huaR y&=&\huaR_0(x)\cdot y-y\cdot \huaR_0(x)-x\cdot y,\\
      x\cdot_\huaR a&=&\huaR_0(x)\cdot a-a\cdot \huaR_0(x)-x\cdot a,\\
      a\cdot_\huaR x&=&\huaR_1(a)\cdot x-x\cdot \huaR_1(a)-a\cdot x,\quad x,y\in A_0,a\in A_{-1}.
    \end{eqnarray}
    Then, $(A_0,A_{-1},\dM,\cdot_\huaR)$ is a strict pre-Lie  $2$-algebra.
\end{pro}
\begin{proof}
	Let $\huaA$ be a strict associative $2$-algebra. By Conditions \rm($a_1$)-\rm($a_3$) of Definition \ref{defi:2-term ass}, Conditions \rm($a_1$)-\rm($a_3$) of Definition \ref{defi:2-term pre} hold. 
	
	By Conditions $\rm(i)$ of Definition \ref{a Rota-Baxter operator} and $(b_1)$ of Definition \ref{defi:2-term ass}, we have
	\begin{eqnarray*}
		&&x_0\cdot_\huaR(x_1\cdot_\huaR x_2)-(x_0\cdot_\huaR x_1)\cdot_\huaR x_2-x_1\cdot_\huaR(x_0\cdot_\huaR x_2)+(x_1\cdot_\huaR x_0)\cdot_\huaR x_2\\
		&=&x_0\cdot_\huaR(\huaR_0(x_1)\cdot x_2-x_2\cdot \huaR_0(x_1)-x_1\cdot x_2)-(\huaR_0(x_0)\cdot x_1-x_1\cdot \huaR_0(x_0)-x_0\cdot x_1)\cdot_\huaR x_2\\
		&&-x_1\cdot_\huaR(\huaR_0(x_0)\cdot x_2-x_2\cdot \huaR_0(x_0)-x_0\cdot x_2)+(\huaR_0(x_1)\cdot x_0-x_0\cdot \huaR_0(x_1)-x_1\cdot x_0)\cdot_\huaR x_2\\
		&=&\huaR_0(x_0)\cdot(\huaR_0(x_1)\cdot x_2)-\huaR_0(x_0)\cdot(x_2\cdot \huaR_0(x_1))-\huaR_0(x_0)\cdot(x_1\cdot x_2)-(\huaR_0(x_1)\cdot x_2)\cdot \huaR_0(x_0)\\
		&&+(x_2\cdot \huaR_0(x_1))\cdot \huaR_0(x_0)+(x_1\cdot x_2)\cdot \huaR_0(x_0)-x_0\cdot(\huaR_0(x_1)\cdot x_2)+x_0\cdot(x_2\cdot \huaR_0(x_1))\\
		&&+x_0\cdot(x_1\cdot x_2)-\huaR_0(\huaR_0(x_0)\cdot x_1)\cdot x_2+\huaR_0(x_1\cdot \huaR_0(x_0))\cdot x_2+\huaR_0(x_0\cdot x_1)\cdot x_2\\
		&&+x_2\cdot\huaR_0(\huaR_0(x_0)\cdot x_1)-x_2\cdot\huaR_0(x_1\cdot \huaR_0(x_0))-x_2\cdot\huaR_0(x_0\cdot x_1)+(\huaR_0(x_0)\cdot x_1)\cdot x_2\\
		&&-(x_1\cdot \huaR_0(x_0))\cdot x_2-(x_0\cdot x_1)\cdot x_2-\huaR_0(x_1)\cdot(\huaR_0(x_0)\cdot x_2)+\huaR_0(x_1)\cdot(x_2\cdot \huaR_0(x_0))\\
		&&+\huaR_0(x_1)\cdot(x_0\cdot x_2)+(\huaR_0(x_0)\cdot x_2)\cdot \huaR_0(x_1)-(x_2\cdot \huaR_0(x_0))\cdot \huaR_0(x_1)-(x_0\cdot x_2)\cdot \huaR_0(x_1)\\
		&&+x_1\cdot(\huaR_0(x_0)\cdot x_2)-x_1\cdot(x_2\cdot \huaR_0(x_0))-x_1\cdot(x_0\cdot x_2)+\huaR_0(\huaR_0(x_1)\cdot x_0)\cdot x_2\\
		&&-\huaR_0(x_0\cdot \huaR_0(x_1))\cdot x_2-\huaR_0(x_1\cdot x_0)\cdot x_2-x_2\cdot\huaR_0(\huaR_0(x_1)\cdot x_0)+x_2\cdot\huaR_0(x_0\cdot \huaR_0(x_1))\\
		&&+x_2\cdot\huaR_0(x_1\cdot x_0)-(\huaR_0(x_1)\cdot x_0)\cdot x_2+(x_0\cdot \huaR_0(x_1))\cdot x_2+(x_1\cdot x_0)\cdot x_2\\
		&=&\huaR_0(x_0)\cdot(\huaR_0(x_1)\cdot x_2)+(x_2\cdot \huaR_0(x_1))\cdot \huaR_0(x_0)-\huaR_0(x_1)\cdot(\huaR_0(x_0)\cdot x_2)-(x_2\cdot \huaR_0(x_0))\cdot \huaR_0(x_1)\\
		&&-(\huaR_0(x_0)\cdot\huaR_0(x_1))\cdot x_2-x_2\cdot (\huaR_0(x_1)\cdot \huaR_0(x_0))+(\huaR_0(x_1)\cdot\huaR_0(x_0))\cdot x_2+x_2\cdot (\huaR_0(x_0)\cdot \huaR_0(x_1))\\
		&=&0.
	\end{eqnarray*}
	
	Similarly, we have 
		$$x_0\cdot_\huaR(x_1\cdot_\huaR a)-(x_0\cdot_\huaR x_1)\cdot_\huaR a-x_1\cdot_\huaR(x_0\cdot_\huaR a)+(x_1\cdot_\huaR x_0)\cdot_\huaR a=0,$$
		$$a\cdot_\huaR(x_1\cdot_\huaR x_2)-(a\cdot_\huaR x_1)\cdot_\huaR x_2-x_1\cdot_\huaR(a\cdot_\huaR x_2)+(x_1\cdot_\huaR a)\cdot_\huaR x_2=0.$$
	
	Therefore, $(A_0,A_{-1},\dM,\cdot_\huaR)$ is a strict pre-Lie  $2$-algebra.
\end{proof}

\begin{defi}
A strict associative $2$-algebra $(A_0, A_{-1}, \dM, \cdot)$ is called a {\bf strict commutative associative $2$-algebra} if for $x,y\in A_0$ and $a\in A_{-1}$, 
$$x\cdot y=y\cdot x,\quad x\cdot a=a\cdot x.$$
\end{defi}

\begin{defi}
Let $\huaA$ be a strict commutative associative $2$-algebra. A pair $(D_0,D_1)$, where $D_0:A_0\longrightarrow A_0,D_1:A_{-1}\longrightarrow A_{-1}$ is a chain map, is called a  {\bf derivation} on $\huaA$  if for all $x,y\in A_0$ and $a\in A_{-1}$, the following conditions are satisfied:
  \begin{itemize}
    \item[\rm(i)] $D_0(x)\cdot y+x\cdot D_0(y)=D_0(x\cdot y);$

    \item[\rm(ii)] $D_0(x)\cdot a+x\cdot D_1(a)=D_1(x\cdot a).$
    \end{itemize}
\end{defi}

\begin{pro}
Let $(D_0,D_1)$ be a derivation on a strict commutative associative $2$-algebra $\huaA$. 
Define a degree $0$ multiplication $\cdot_D:A_i\otimes A_j\longrightarrow A_{i+j}$, $-1\leq i+j\leq 0$, on $\huaA$ by
    \begin{eqnarray}
      x\cdot_D y&=&x\cdot D_0(y)+cx\cdot y,\\
      x\cdot_D a&=&x\cdot D_1(a)+cx\cdot a,\\
      a\cdot_D x&=&a\cdot D_0(x)+ca\cdot x,
    \end{eqnarray}
    where $ x,y\in A_0,a\in A_{-1}$ and $c$ is a fixed constant in $\K$. 
    Then, $(A_0,A_{-1},\dM,\cdot_D)$ is a strict pre-Lie  $2$-algebra.
\end{pro}

\begin{proof}
	Let $\huaA$ be a strict commutative associative $2$-algebra. By Conditions \rm($a_1$)-\rm($a_3$) of Definition \ref{defi:2-term ass}, Conditions \rm($a_1$)-\rm($a_3$) of Definition \ref{defi:2-term pre} follow. Since $\huaA$ is commutative and Condition $(b_1)$ of Definition \ref{defi:2-term ass} holds, we have
	\begin{eqnarray*}
	&&x_0\cdot_D(x_1\cdot_D x_2)-(x_0\cdot_D x_1)\cdot_D x_2-x_1\cdot_D(x_0\cdot_D x_2)+(x_1\cdot_D x_0)\cdot_D x_2\\
	&=&x_0\cdot D_0(x_1 \cdot D_0(x_2))+cx_0\cdot(x_1 \cdot D_0(x_2))+x_0\cdot D_0(cx_1\cdot x_2)+cx_0\cdot(cx_1\cdot x_2)\\
	&&-(x_0 \cdot D_0(x_1))\cdot D_0(x_2)-c(x_0 \cdot D_0(x_1))\cdot x_2-(cx_0\cdot x_1)\cdot D_0(x_2)-c(cx_0\cdot x_1)\cdot x_2\\
	&&-x_1\cdot D_0(x_0 \cdot D_0(x_2))-cx_1\cdot (x_0 \cdot D_0(x_2))-x_1\cdot D_0(cx_0\cdot x_2)-cx_1\cdot (cx_0\cdot x_2)\\
	&&+(x_1 \cdot D_0(x_0))\cdot D_0(x_2)+c(x_1 \cdot D_0(x_0))\cdot x_2+(cx_1\cdot x_0)\cdot D_0(x_2)+c(cx_1\cdot x_0)\cdot x_2=0.
	\end{eqnarray*}

Since $\huaA$ is commutative and Condition $(b_2)$ of Definition \ref{defi:2-term ass} holds, we have
	\begin{eqnarray*}
		&&x_0\cdot_D(x_1\cdot_D a)-(x_0\cdot_D x_1)\cdot_D a-x_1\cdot_D(x_0\cdot_D a)+(x_1\cdot_D x_0)\cdot_D a\\
		&=&x_0\cdot D_1(x_1\cdot D_1(a))+x_0\cdot D_1(cx_1\cdot a)+cx_0\cdot (x_1\cdot D_1(a))+cx_0\cdot(cx_1\cdot a)\\
		&&-(x_0\cdot D_0(x_1))\cdot D_1(a)-(cx_0\cdot x_1)\cdot D_1(a)+c(x_0\cdot D_0(x_1))\cdot a+c(cx_0\cdot x_1)\cdot a\\
		&&-x_1\cdot D_1(x_0\cdot D_1(a))-x_1\cdot D_1(cx_0 \cdot a)-cx_1\cdot (x_0\cdot D_1(a))-cx_1\cdot (cx_0 \cdot a)\\
		&&+(x_1\cdot D_0(x_0))\cdot D_1(a)+(cx_1 \cdot x_0)\cdot D_1(a)+c(x_1\cdot D_0(x_0))\cdot a+c(cx_1 \cdot x_0)\cdot a=0.
	\end{eqnarray*}

	Similarly, since $\huaA$ is commutative and Conditions $(b_3)$ and $(b_4)$ of Definition \ref{defi:2-term ass} hold, we can get 
    $$a\cdot_D(x_1\cdot_D x_2)-(a\cdot_D x_1)\cdot_D x_2-x_1\cdot_D(a\cdot_D x_2)+(x_1\cdot_D a)\cdot_D x_2=0.$$
	
	Therefore, $(A_0,A_{-1},\dM,\cdot_D)$ is a strict pre-Lie $2$-algebra.
\end{proof}

\subsection{Representations of strict pre-Lie 2-algebras}
 
 \begin{defi}\label{Representations of strict pre-Lie 2-algebras}
 Let $\huaA=(A_0, A_{-1}, \dM, \cdot)$ be a strict pre-Lie $2$-algebra, $\huaV:V_{-1} \stackrel{\partial}{\longrightarrow} V_0$ be a 2-term complex of vector spaces. A {\bf strict representation} of $\huaA$ on $\huaV$ consists of $\rho=(\rho_0, \rho_1)$ and $\mu=(\mu_0, \mu_1)$, where $\rho=(\rho_0, \rho_1)$ is a strict representation of the strict Lie $2$-algebra $\huaG(\huaA)$ on $\huaV$ and $\mu=(\mu_0, \mu_1)$ is a chain map from $2$-term complex of vector spaces $\huaA$ to $\End(\huaV)$, such that for all $x, y, z \in A_0, a \in A_{-1}$, the following equalities are satisfied:
 \begin{eqnarray}
\label{eq:rep1}\mu_0(x\cdot y)+\mu_0(y)\circ\rho_0(x)-\rho_0(x)\circ\mu_0(y)-\mu_0(y)\circ\mu_0(x)&=&0,\\
\label{eq:rep2}\mu_1(x\cdot a)-\rho_0(x)\circ\mu_1(a)+\mu_1(a)\circ\rho_0(x)-\mu_1(a)\circ\mu_0(x)&=&0,\\
\label{eq:rep3}\mu_1(a\cdot x)-\rho_1(a)\circ\mu_0(x)+\mu_0(x)\circ\rho_1(a)-\mu_0(x)\circ\mu_1(a)&=&0.
 \end{eqnarray}
 \end{defi}

\begin{pro}\label{pro:rep-pre2-Lie2}
Let $\huaA=(A_0, A_{-1}, \dM, \cdot)$ be a strict pre-Lie $2$-algebra and  $(\huaV;(\rho_0,\rho_1),(\mu_0,\mu_1))$ be its strict representation. Then $(\huaV;\rho_0-\mu_0,\rho_1-\mu_1)$ is a strict representation of the strict Lie $2$-algebra $\huaG(\huaA)$.
\end{pro}
\begin{proof}
	Let $\huaA=(A_0, A_{-1}, \dM, \cdot)$ be a strict pre-Lie $2$-algebra and $(\huaV;(\rho_0,\rho_1),(\mu_0,\mu_1))$ be its strict representation. Then, $\rho=(\rho_0, \rho_1)$ is a strict representation of the strict Lie $2$-algebra $\huaG(\huaA)$ on $\huaV$ and $\mu=(\mu_0, \mu_1)$ is a chain map from $2$-term complex of vector spaces $\huaA$ to $\End(\huaV)$. 
	
	Through the definition of the strict representation of the strict Lie $2$-algebra $\huaG(\huaA)$, we have 
	\begin{itemize}
		\item[\rm(i)] $\rho_0\circ \dM=\delta\circ \rho_1,$
		\item[\rm(ii)] $\rho_0\mathfrak{l}_2(x, y)-[\rho_0(x), \rho_0(y)]=0,$
		\item[\rm(iii)] $\rho_1\mathfrak{l}_2(x, h)-[\rho_0(x), \rho_1(h)]=0, \quad x,y\in A_0, h\in A_1.$
	\end{itemize}
	
	Because $\mu=(\mu_0, \mu_1)$ is a chain map from $2$-term complex of vector spaces $\huaA$ to $\End(\huaV)$, we know $\mu_0\circ \dM(h)=\delta\circ\mu_1(h).$ By \rm(i), $$(\rho_0-\mu_0)\circ \dM=\delta\circ (\rho_1-\mu_1)$$ holds.
	
	From \eqref{eq:rep1}, we can get $\mu_0(x\cdot y-y\cdot x)-[\mu_0(x), \rho_0(y)]-[\rho_0(x), \mu_0(y)]+[\mu_0(x), \mu_0(y)]=0.$
	By \rm(ii), $$(\rho_0-\mu_0)\mathfrak{l}_2(x, y)-[(\rho_0-\mu_0)(x), (\rho_0-\mu_0)(y)]=0$$ holds.
	
	From \eqref{eq:rep2} and \eqref{eq:rep3}, we can get $\mu_1(x\cdot h-h\cdot x)-[\rho_0(x), \mu_1(h)]-[\mu_1(h), \mu_0(x)]+[\rho_1(h), \mu_0(x)]=0.$
	By \rm(iii), $$(\rho_1-\mu_1)\mathfrak{l}_2(x, h)-[(\rho_0-\mu_0)(x), (\rho_1-\mu_1)(h)]=0$$ holds.
	
	Then, $(\huaV;\rho_0-\mu_0,\rho_1-\mu_1)$ is a strict representation of the strict Lie $2$-algebra $\huaG(\huaA)$.
\end{proof}

 \begin{pro}
 Let $(\huaV;\rho,\mu)$ be a strict representation of the strict pre-Lie $2$-algebra $\huaA=(A_0, A_{-1}, \dM, \cdot)$.   Considering $2$-term complex $\huaA\oplus\huaV:A_{-1}\oplus V_{-1}\stackrel{\dM+\partial }{\longrightarrow}A_0\oplus V_0$, and for all $x, y \in A_0$, $a, b \in A_{-1}$, $u, v \in V_0$, $m, n \in V_{-1}$, define bilinear maps $\ast: (A_0\oplus V_0) \oplus (A_0\oplus V_0) \longrightarrow A_{0}\oplus V_{0}$, $\ast: (A_0\oplus V_0) \oplus (A_{-1}\oplus V_{-1}) \longrightarrow A_{1}\oplus V_{1}$ and $\ast: (A_{-1}\oplus V_{-1}) \oplus (A_0\oplus V_0) \longrightarrow A_{1}\oplus V_{1}$ by
 \begin{eqnarray}
(x+u)\ast(y+v)&=&x\cdot y+\rho_0(x)v+\mu_0(y)u,\\
(x+u)\ast (a+m)&=&x\cdot a+\rho_0(x)m+\mu_1(a)u,\\
(a+m)\ast (y+v)&=&a\cdot y+\rho_1(a)v+\mu_0(y)m.
 \end{eqnarray}
 Then $(A_0\oplus V_0, A_{-1}\oplus V_{-1}, \dM+\partial, \ast)$ is a strict pre-Lie $2$-algebra, which is called a {\bf semidirect product strict pre-Lie $2$-algebra}  and denote it by $\huaA\ltimes_{(\rho, \mu)} \huaV$.
 \end{pro}
 \begin{proof}
 	For all $x, y, z \in A_0$, $a, b \in A_{-1}$, $u, v, w \in V_0$, $m, n \in V_{-1}$. By the fact that  $\rho=(\rho_0, \rho_1)$ and $\mu=(\mu_0, \mu_1)$ are cochain maps, we have
 	\begin{eqnarray*}
 	(\dM+\partial)((x+u)\ast(a+m))&=& d(x\cdot a)+\partial(\rho_0(x)m)+\partial(\mu_1(a)u))\\
 	&=& x \cdot d(a)+\rho_0(x)(\partial(m))+\mu_0(d(a))(u)\\
 	&=& (x+u)\ast (d(a)+\partial(m))\\
 	&=& (x+u)\ast (\dM+\partial)(a+m),
 	\end{eqnarray*}
 which implies that 
 $$	(\dM+\partial)((x+u)\ast(a+m))= (x+u)\ast (\dM+\partial)(a+m).$$
 
 Similarly, we have 
 \begin{eqnarray*}
 	(\dM+\partial)((a+m)\ast(x+u))&=&((\dM+\partial)(a+m))\ast(x+u),\\ ((\dM+\partial)(a+m))\ast(b+n)&=&(a+m)\ast((\dM+\partial)(b+n)).
 	\end{eqnarray*}
 
 	By the fact that $\rho=(\rho_0, \rho_1)$ is a strict representation of a strict Lie $2$-algebra $\huaG(\huaA)$ and \eqref{eq:rep1}, we have
 	\begin{eqnarray*}
 	&&(x+u)\ast((y+v)\ast(z+w))-((x+u)\ast(y+v))\ast(z+w)\\
 	&&-(y+v)\ast((x+u)\ast(z+w))+((y+v)\ast(x+u))\ast(z+w)\\
 	&=&x\cdot(y\cdot z)-(x\cdot y)\cdot z-y\cdot(x\cdot z)+(y\cdot x)\cdot z\\
 	&&+\rho_0(x)\rho_0(y)w-\rho_0(x\cdot y)w-\rho_0(y)\rho_0(x)w+\rho(y\cdot x)w\\
 	&&+\rho_0(x)\mu_0(z)v-\mu_0(z)\rho_0(x)v-\mu_0(x\cdot z)v+\mu_0(z)\mu_0(x)v\\
 	&&+\mu_0(y\cdot z)u-\mu_0(z)\mu_0(y)u-\rho_0(y)\mu_0(z)u+\mu_0(z)\rho_0(y)u\\
 	&=&0.
 	\end{eqnarray*}
 	Similarly, we have
 	\begin{eqnarray*}
 	&&(x+u)\ast((y+v)\ast(a+m))-((x+u)\ast(y+v))\ast(a+m)\\
 	&&-(y+v)\ast((x+u)\ast(a+m))+((y+v)\ast(x+u))\ast(a+m)=0,\\
 	&&(a+m)\ast((y+v)\ast(z+w))-((a+m)\ast(y+v))\ast(z+w)\\
 	&&-(y+v)\ast((a+m)\ast(z+w))+((y+v)\ast(a+m))\ast(z+w)=0.
 	\end{eqnarray*}
 	Therefore, $(A_0\oplus V_0, A_{-1}\oplus V_{-1}, \dM+\partial, \ast)$ is a strict pre-Lie $2$-algebra.
 \end{proof}

 \begin{pro}
 Let $\huaA=(A_0, A_{-1}, \dM, \cdot)$ be a strict pre-Lie $2$-algebra and $(\huaV;(\rho_0,\rho_1),(\mu_0,\mu_1))$ be its strict representation. Then $(\huaV^*; (\rho_0^*-\mu_0^*, \rho_1^*-\mu_1^*), (-\mu_0^*, -\mu_1^*))$ is also a strict representation of $\huaA$.
 \end{pro}
 \begin{proof}
By Proposition \ref{pro:rep-pre2-Lie2}, $(\huaV;\rho_0-\mu_0,\rho_1-\mu_1)$ is a strict representation of the strict Lie $2$-algebra $\huaG(\huaA)$. Thus  $(\huaV^*;\rho_0^*-\mu_0^*, \rho_1^*-\mu_1^*)$ is a strict representation of $\huaG(A)$.

For any $a \in A_{-1}$, $n\in V_{-1}$, $m^*\in V^*_1$, we have
 \begin{eqnarray*}
     \left\langle (-\mu_0^*)(\dM a)m^*, n\right\rangle&=&\left\langle m^*, \mu_0(\dM a)n\right\rangle= \left\langle m^*, \partial\mu_1( a)n\right\rangle
=\left\langle m^*, \partial\mu_1( a)n\right\rangle\\
     &=&\left\langle(-\mu^*_1( a)) \partial^*(m^*),n\right\rangle
     =\left\langle(\delta^*(-\mu^*_1( a)))(m^*),n\right\rangle,     \end{eqnarray*}
     which implies that 
     $$(-\mu_0^*)(\dM a)m^*=\delta^*(-\mu^*_1( a)))(m^*).$$
 Similar, for  $v^*\in V^*_0$, we have
 $$(-\mu_0^*)(\dM a)v^*=\delta^*(-\mu^*_1( a)))(v^*).$$
Thus, $(-\mu_0^*, -\mu_1^*)$ is a chain map from $\huaA$ to $\End(\huaV^*)$. 

 For  $x,y\in A_0,v\in V_0,u^*\in V^*_0$,  by \eqref{eq:rep1}, we have
\begin{eqnarray*}
&&\left\langle -\mu_0^*(x\cdot y)u^*+(-\mu_0^*(y))(\rho_0^*-\mu_0^*)(x)u^*-(\rho_0^*-\mu_0^*)(x)(-\mu_0^*)(y)u^*-(-\mu_0^*(y))(-\mu_0^*(x))u^*, v\right\rangle\\
&=&\left\langle u^*, \mu_0(x\cdot y)v-\rho_0(x)\mu_0(y)v+\mu_0(y)\rho_0(x)-\mu_0(y)\mu_0(x)v\right\rangle=0,
\end{eqnarray*}
which implies that 
$$-\mu_0^*(x\cdot y)+(-\mu_0^*(y))(\rho_0^*-\mu_0^*)(x)-(\rho_0^*-\mu_0^*)(x)(-\mu_0^*)(y)-(-\mu_0^*(y))(-\mu_0^*(x))=0.$$

 For  $x\in A_0,a\in A_{-1}, u\in V_0,m^*\in V^*_1$,  by \eqref{eq:rep2}, we have
\begin{eqnarray*}
&&\left\langle -\mu_1^*(x\cdot a)m^*-(\rho_0^*-\mu_0^*)(x)(-\mu_1^*(a))m^*+(-\mu_1^*(a))\circ(\rho_0^*-\mu_0^*)(x)m^*-(-\mu_1^*(a))(-\mu_0^*(x))m^*, u\right\rangle\\
&=&\left\langle m^*, \mu_1(x\cdot a)u+\mu_1(a)\rho_0(x)u-\mu_1(a)\mu_0(x)u-\rho_0(x)\mu_1(a)u\right\rangle=0,
\end{eqnarray*}
which implies that 
$$-\mu_1^*(x\cdot a)-(\rho_0^*-\mu_0^*)(x)(-\mu_1^*(a))+(-\mu_1^*(a))\circ(\rho_0^*-\mu_0^*)(x)-(-\mu_1^*(a))(-\mu_0^*(x))=0.$$

 For  $x\in A_0,a\in A_{-1}, u\in V_0,m^*\in V^*_1$,  by \eqref{eq:rep3}, we have
\begin{eqnarray*}
&&\left\langle -\mu_1^*(a\cdot x)m^*-(\rho_1^*-\mu_1^*)(a)\circ(-\mu_0^*)(x)m^*+(-\mu_0^*(x))\circ(\rho_1^*-\mu_1^*)(a)m^*-(-\mu_0^*(x))(-\mu_1^*(a))m^*, u\right\rangle\\
&=&\left\langle m^*, \mu_1(a\cdot x)u+\mu_0(x)\rho_1(a)u-\mu_0(x)\mu_1(a)u-\rho_1(a)\mu_0(x)u\right\rangle=0,
\end{eqnarray*}
which implies that 
$$-\mu_1^*(a\cdot x)-(\rho_1^*-\mu_1^*)(a)\circ(-\mu_0^*)(x)+(-\mu_0^*(x))\circ(\rho_1^*-\mu_1^*)(a)-(-\mu_0^*(x))(-\mu_1^*(a))=0.$$
Therefore, $(\huaV^*; (\rho_0^*-\mu_0^*, \rho_1^*-\mu_1^*), (-\mu_0^*, -\mu_1^*))$ is also a strict representation of $\huaA$.
 \end{proof}

 \begin{ex}\label{ex1}
 	Let $\huaA=(A_0, A_{-1}, \dM, \cdot)$ be a strict pre-Lie $2$-algebra. Define $L_0,R_0:A_0\longrightarrow \End_{\dM}^0(\huaA)=\End(A_0)\oplus \End(A_{-1})$ and $L_1, R_1:A_{-1}\longrightarrow \End^{-1}(\huaA)=\Hom(A_0, A_{-1})$ by 
 \begin{eqnarray*}
 	L_0(x)y=x\cdot y, \quad L_0(x)a=x\cdot a,\quad R_0(x)y=y\cdot x, \\
 	R_0(x)a=a\cdot x,\quad L_1(a)x=a\cdot x,\quad R_1(a)x=x\cdot a,
 \end{eqnarray*}
 for any $x, y \in A_0$, $a \in A_{-1}$. Then $(\huaA; (L_0, L_1), (R_0, R_1))$ is a strict representation of $\huaA$, which is called a {\bf  regular representation}. Furthermore, $(\huaA^*; (\ad^*_0, \ad^*_1), (-R^*_0, -R^*_1))$ is also a strict representation of $\huaA$, which is called a {\bf  coregular representation}.
 \end{ex}
 
 \emptycomment{\begin{pro}\label{pro}
  Let $\huaA=(A_0, A_{-1}, \dM, \cdot)$ and $\huaA^*=(A_{-1}^*, A_0^*, \dM^*, \circ)$ be two strict pre-Lie $2$-algebras. Denote $\partial: A_{-1}\oplus A_0^* \longrightarrow A_0\oplus A_{-1}^*$ by $\partial(a+x^*)=\dM a+\dM^*x^*$ for any $a \in A_{-1}, x^* \in A_0^*$. Denote $\ast: (A_0\oplus A_{-1}^*) \otimes (A_0\oplus A_{-1}^*)\longrightarrow (A_0\oplus A_{-1}^*)$, $\ast: (A_0\oplus A_{-1}^*) \otimes (A_{-1}\oplus A_0^*)\longrightarrow (A_{-1}\oplus A_0^*)$ and $\ast: (A_{-1}\oplus A_0^*) \otimes (A_0\oplus A_{-1}^*)\longrightarrow (A_{-1}\oplus A_0^*)$ satisfying 
  $$(x+a^*)\ast(y+b^*)=x\cdot y+a^*\circ b^*+\ad_0^*(x)b^*-\huaR_0^*(b^*)x+\mathfrak{ad}_0^*(a^*)y-R_0^*(y)a^*,$$
  $$(x+a^*)\ast (b+y^*)=x\cdot b+a^*\circ y^*+\ad_0^*(x)y^*-\huaR_1^*(y^*)x+\mathfrak{ad}_0^*(a^*)b-R_1^*(b)a^*,$$
  $$(b+y^*)\ast (x+a^*)=b\cdot x+y^*\circ a^*+\ad_1^*(b)a^*-\huaR_0^*(a^*)b+\mathfrak{ad}_1^*(y^*)x-R_0^*(x)y^*,$$
  for any $x+a^*$, $y+b^* \in A_0\oplus A_{-1}^*$, $b+y^* \in A_{-1}\oplus A_0^*$.
  Then $(A_0\oplus A_{-1}^*, A_{-1}\oplus A_0^*, \partial, \ast)$ is a strict pre-Lie $2$-algebra.
 \end{pro}}

 \section{Para-K\"ahler strict Lie 2-algebras, Manin triples and matched pairs of strict pre-Lie 2-algebras}\label{sec:Manin triple}
 In this section, we fist give the categorification of para-K\"ahler Lie algebras and introduce the notion of a para-K\"ahler strict Lie $2$-algebra. Then we introduce the special para-K\"ahler strict Lie $2$-algebras, Manin triples and matched pairs of strict pre-Lie $2$-algebras. The equivalences between those structures are established.
 \subsection{Para-K\"ahler strict Lie 2-algebra and Manin triples of strict pre-Lie 2-algebras}
 A nondegenerate degree $1$ graded antisymmetric bilinear form $\omega$ on a strict pre-Lie $2$-algebra $\huaA$ is called {\bf invariant} if 
 \begin{eqnarray}
\label{eq:invariance 1} \omega(\dM \alpha,\beta)&=&(-1)^{|\alpha||\beta|}\omega(\dM \beta,\alpha),\\
\label{eq:invariance 2}\omega(\alpha\cdot \beta,\gamma)&=&(-1)^{|\beta||\gamma|}\omega([\alpha,\gamma],\beta),\quad \alpha,\beta,\gamma\in \huaA.
\end{eqnarray}

The nondegenerate degree $1$ graded antisymmetric bilinear form $\omega$ means that 
$\omega$ induces the isomorphisms $A_0\simeq (A_{-1})^*$ and  $A_{-1}\simeq (A_{0})^*$. Furthermore, the bilinear form $\omega$ is given by
\begin{equation}
\omega(x+a, y+b)=\omega(x,b)+\omega(a,y),\quad x,y\in A_0,a,b\in A_{-1}.
\end{equation}

\begin{defi}
  A {\bf quadratic strict pre-Lie $2$-algebra} is a pair $(\huaA,\omega)$, where $\huaA$ is a strict pre-Lie $2$-algebra and $\omega$ is a nondegenerate graded  degree $1$ invariant bilinear form.
\end{defi}

 \begin{defi}
 	A {\bf Manin triple of strict pre-Lie $2$-algebras} is a triple $((\huaA, \omega), \huaA_1, \huaA_2)$, where 
 	$(\huaA, \omega)$ is an even dimensional quadratic strict pre-Lie $2$-algebra, $\huaA_1$ and $\huaA_2$ are strict sub-pre-Lie $2$-algebras of $\huaA$, both isotropic with respect to $\omega$ in the sense of 
    $$\omega(\alpha_1,\beta_1)=0,\quad \omega(\alpha_2,\beta_2)=0,\quad \alpha_1,\beta_1\in \huaA_1,\alpha_2,\beta_2\in \huaA_2,$$
    and $\huaA=\huaA_1 \oplus \huaA_2$ as graded vector spaces.
 \end{defi}
 
  Let $\huaA=(A_0, A_{-1}, \dM, \cdot)$ and $\huaA^*=(A_{-1}^*, A_0^*, \dM^*, \circ)$ be two strict pre-Lie $2$-algebras. On the direct sum of complexes, $\partial: A_{-1}\oplus A_0^* \stackrel{\dM+\dM^*}{\longrightarrow} A_0\oplus A_{-1}^*$, there is a nondegenerate invariant degree $1$ graded antisymmetric bilinear form $\omega$ given by, 
\begin{equation}\label{omega}
\omega(x+a+x^*+a^*, y+b+y^*+b^*)=\left\langle x^*, y\right\rangle+\left\langle a^*, b\right\rangle-\left\langle x, y^*\right\rangle-\left\langle a, b^*\right\rangle,
\end{equation}
which is called {\bf standard antisymmetric bilinear form} on $\huaA \oplus \huaA^*$. We can introduce an operation $\ast$ on $\huaA\oplus\huaA^*$ such that $\omega$ is invariant, as follows: 
  \begin{eqnarray}
 \label{eq:std 1}   (x+a^*)\ast(y+b^*)&=&x\cdot y+a^*\circ b^*+\ad_0^*(x)b^*-\huaR_0^*(b^*)x+\mathfrak{ad}_0^*(a^*)y-R_0^*(y)a^*,\\
  \label{eq:std 2}   (x+a^*)\ast (b+y^*)&=&x\cdot b+a^*\circ y^*+\ad_0^*(x)y^*-\huaR_1^*(y^*)x+\mathfrak{ad}_0^*(a^*)b-R_1^*(b)a^*,\\
   \label{eq:std 3}  (b+y^*)\ast (x+a^*)&=&b\cdot x+y^*\circ a^*+\ad_1^*(b)a^*-\huaR_0^*(a^*)b+\mathfrak{ad}_1^*(y^*)x-R_0^*(x)y^*,
  \end{eqnarray}
  where $x+a^*$, $y+b^* \in A_0\oplus A_{-1}^*$, $b+y^* \in A_{-1}\oplus A_0^*$.
 
  If $\huaA \oplus \huaA^*=(A_0\oplus A_{-1}^*, A_{-1}\oplus A_0^*, \partial=\dM+\dM^*, \ast)$ is a strict pre-Lie $2$-algebra (in this case, $\huaA$ and $\huaA^*$ are sub-pre-Lie $2$-algebras naturally), then we obtain a Manin triple $(\huaA \oplus \huaA^*; \huaA, \huaA^*)$ with respect to the standard bilinear form $\omega$, which we called the {\bf standard Manin triple} of strict pre-Lie $2$-algebras.

  \begin{pro}
  	Any Manin triple of strict pre-Lie $2$-algebras $(\huaA; \huaA_1, \huaA_2)$ with respect to a nondegenerate invariant degree $1$ graded antisymmetric bilinear form $\omega$ is isomorphic to a standard Manin triple $(\huaA_1 \oplus \huaA_1^*; \huaA_1, \huaA_1^*)$.
  \end{pro}
\begin{proof}
The nondegeneracy of $\omega$ implies that $\huaA_2$ is isomorphic to $\huaA_1^*$. Furthermore, because $\omega$ is the degree $1$, we have $(\huaA_2)_0\simeq (\huaA_1)^*_{-1}$ and $(\huaA_2)_{-1}\simeq (\huaA_1)^*_{0}$. The invariancy of
$\omega$ implies that the operation $\ast$ must be given by \eqref{eq:std 1}-\eqref{eq:std 3}.
\end{proof}

\begin{defi}
Let $(\g=\g_0\oplus \g_{-1};\omega=(\omega_1,\omega_2))$ be a symplectic strict Lie $2$-algebra. A strict sub-Lie $2$-algebra $\h=\h_0\oplus \h_1$ of $(\g;\omega)$ is called {\bf Lagrangian} if $\h^\perp=\h$, where
$$\h^\perp=\{u\in \g\mid \omega(u,v)=0,\quad v\in \h\}.$$ Furthermore, a symplectic strict Lie $2$-algebra  $(\g;\omega)$ is called a {\bf para-K\"ahler strict Lie $2$-algebra} if $\g$ is a direct sum of the underlying graded vector spaces of two Lagrangian strict sub-Lie $2$-algebra $\g^+$ and $\g^-$. A para-K\"ahler strict Lie $2$-algebra is called a {\bf special para-K\"ahler strict Lie $2$-algebra} if $\omega_1=0$, that is, $\omega\in \Hom(\g_0\wedge \g_{-1},\K)$. We denote a (special) para-K\"ahler strict Lie $2$-algebra by $((\g,\omega),\g^+,\g^-)$.
\end{defi}

\begin{thm}
 Let $((\huaA=(A_0,A_{-1},\dM,\ast), \omega), \huaA_1, \huaA_2)$ be a Manin triple of strict pre-Lie $2$-algebras. Then $((\huaA^c=(A_0,A_{-1},\dM,[-,-]^c]), \omega), \huaA^c_1, \huaA^c_2)$ is a special para-K\"ahler strict Lie $2$-algebra, where the bracket $[-,-]^c:A_i\times A_j\rightarrow A_{i+j},~-1\leq i+j\leq 0$ is given by \eqref{eq:l21} and \eqref{eq:l22}.
 
 Conversely, if $((\g=(\g_0,\g_{-1},\dM,[-,-]),\omega),\g^+,\g^-)$ is a special para-K\"ahler strict Lie $2$-algebra, then $((\g=(\g_0,\g_{-1},\dM,\ast),\omega),\g^+,\g^-)$ is a Manin triple of strict pre-Lie $2$-algebras, where the operation $\ast:\g_i\times \g_j\rightarrow \g_{i+j}~-1\leq i+j\leq 0$ is given by 
\begin{equation}
\omega(u\ast v,w)=(-1)^{|v||w|}\omega([u,w],v),\quad u,v,w\in\g.
\end{equation}
\end{thm}
\begin{proof}
Let $\huaA=(A_0,A_{-1},\dM,\ast)$ be a strict pre-Lie $2$-algebra. Then $\huaA^c=(A_0,A_{-1},\dM,[-,-]^c])$ is a strict Lie $2$-algebra. Furthermore,  since $\huaA_1$ and $ \huaA_2$ are strict sub-pre-Lie $2$-algebras, $\huaA^c_1$ and $ \huaA^c_2$ are strict sub-Lie $2$-algebras. By the fact that $\omega$ is a nondegenerate degree $1$ graded antisymmetric bilinear form and $\huaA_1$ and $\huaA_2$ are isotropic with respect to $\omega$, for $x_1,y_1,z_1\in (\huaA_1)_0$ and $x_2,y_2,z_2\in (\huaA_2)_0$, we have $\omega(x_1+x_2,y_1+y_2)=0,$
which implies that 
$$\omega([x_1+x_2,y_1+y_2]^c,z_1+z_2)+\omega([y_1+y_2,z_1+z_2]^c,x_1+x_2)+\omega([z_1+z_2,x_1+x_2]^c,y_1+y_2)=0.$$

For $a_1\in (\huaA_1)_{-1}$ and $a_2\in (\huaA_2)_{-1}$, by \eqref{eq:invariance 2}, we have
\begin{eqnarray*}
&&\omega([x_1+x_2,y_1+y_2]^c,a_1+a_2)+\omega([y_1+y_2,a_1+a_2]^c,x_1+x_2)+\omega([a_1+a_2,x_1+x_2]^c,y_1+y_2)\\
&=&\omega([x_1+x_2,y_1+y_2]^c,a_1+a_2)+\omega((y_1+y_2)\ast (x_1+x_2),a_1+a_2)\\
&&-\omega((x_1+x_2)\ast (y_1+y_2), a_1+a_2)\\
&=&\omega([x_1+x_2,y_1+y_2]^c-(x_1+x_2)\ast (y_1+y_2)+(y_1+y_2)\ast (x_1+x_2),a_1+a_2)\\
&=&0.
\end{eqnarray*}
By \eqref{eq:invariance 1}, we have
$$\omega(x_1+x_2,\dM (a_1+a_2))=0,\quad \omega(\dM(a_1+a_2),b_1+b_2)=\omega(a_1+a_2,\dM(b_1+b_2)).$$
Therefore, $((\huaA^c=(A_0,A_{-1},\dM,[-,-]^c]), \omega), \huaA^c_1, \huaA^c_2)$ is a special para-K\"ahler strict Lie $2$-algebra.

Conversely, if $(\g=(\g_0,\g_{-1},\dM,[-,-]),\omega)$ is a symplectic strict Lie $2$-algebra, by Proposition \ref{pro:symplectic strict L2-preL2}, $\g=(\g_0,\g_{-1},\dM,\ast)$ is a strict pre-Lie $2$-algebra. For $x_+,y_+\in \g_0^+$, since $\g^+$ is a Lagrangian sub-Lie $2$-algebra, we have 
$$\omega(x_+\ast y_+,z_++h_+)=-\omega(y_+,[x_+,z_++h_+])=0,\quad \forall~z_+\in\g_0^+,h_+\in \g_{-1}^+, $$
which implies  $x_+\ast y_+\in \g^+$. If $x_+\ast y_+\notin \g_0^+$, there exists $x_-\in \g_{0}^-$ such that $\omega(x_+\ast y_+,x_-) \neq 0$. But $\omega(x_+\ast y_+,x_-)=-\omega(y_+,[x_+,x_-])=0$ because of $\omega\mid_{\g_0\times\g_0}=0$. This implies that $x_+\ast y_+\in \g_0^+$. Furthermore, we can show that $x_+\ast h_+\in \g_1^+$ for all $x_+\in \g_0^+$ and $h_+\in  \g_1^+$. Thus $\g^+$ is a Lagrangian sub-pre-Lie $2$-algebra. Similarly, $\g^-$ is also a Lagrangian sub-pre-Lie $2$-algebra. The rest can be obtained directly. We omit the details.

\end{proof}

 \subsection{Matched pair of strict pre-Lie 2-algebras}
 In this subsection, we give the notion of matched pair of strict pre-Lie $2$-algebras.
 \begin{thm}
 Let $\huaA=(A_0, A_{-1}, \dM, \cdot)$ and $\huaA'=(A_0', A_{-1}', \dM', \circ)$ be two strict pre-Lie $2$-algebras. Let $\rho=(\rho_0, \rho_1), \mu=(\mu_0, \mu_1): \huaA \longrightarrow \End(\huaA')$ and $\rho'=(\rho_0', \rho_1'), \mu'=(\mu_0', \mu_1'): \huaA' \longrightarrow \End(\huaA)$ be strict representations of $\huaA$ and $\huaA'$ on $\huaA'$ and $\huaA$, respectively, satisfying the following compatibility conditions:
 \begin{eqnarray}\label{matched pair}
 \label{11} \mu_1'(a')(x\cdot y-y\cdot x)&=&x\cdot (\mu_1'(a')y)-y\cdot (\mu_1'(a')x)+\mu_1'(\rho_0(y)a')x-\mu_1'(\rho_0(x)a')y,\\
 \label{22} \mu_1(a)(x'\circ y'-y'\circ x')&=&x'\circ (\mu_1(a)y')-y'\circ (\mu_1(a)x')+\mu_1(\rho_0'(y')a)x'-\mu_1(\rho_0'(x')a)y',\\
 \label{33} \mu_0'(z')(x\cdot y-y\cdot x)&=&x\cdot (\mu_0'(z')y)-y\cdot (\mu_0'(z')x)+\mu_0'(\rho_0(y)z')x-\mu_0'(\rho_0(x)z')y,\\
  \label{44} \mu_0(z)(x'\circ y'-y'\circ x')&=&x'\circ (\mu_0(z)y')-y'\circ (\mu_0(z)x')+\mu_0(\rho_0'(y')z)x'-\mu_0(\rho_0'(x')z)y',\\
 \label{55} \mu_0'(y')(x\cdot a-a\cdot x)&=&x\cdot (\mu_0'(y')a)-a\cdot (\mu_0'(y')x)+\mu_1'(\rho_1(a)y')x-\mu_0'(\rho_0(x)y')a,\\
  \label{66} \mu_0(y)(x'\circ a'-a'\circ x')&=&x'\circ (\mu_0(y)a')-a'\circ (\mu_0(y)x')+\mu_1(\rho_1'(a')y)x'-\mu_0(\rho_0'(x')y)a',\\
  \label{77} \rho_0(x)(y'\circ a')&=&(\rho_0(x)y'-\mu_0(x)y')\circ a'+\rho_0(\mu_0'(y')x-\rho_0'(y')x)a'\\
  \nonumber&&+y'\circ (\rho_0(x)a')+\mu_1(\mu_1'(a')x)y',\\
  \label{88} \rho_0'(x')(y\cdot a)&=&(\rho_0'(x')y-\mu_0'(x')y)\cdot a+\rho_0'(\mu_0(y)x'-\rho_0(y)x')a\\
  \nonumber&&+y\cdot (\rho_0'(x')a)+\mu_1'(\mu_1(a)x')y,\\
  \label{99} \rho_0(x)(a'\circ y')&=&(\rho_0(x)a'-\mu_0(x)a')\circ y'+\rho_1(\mu_1'(a')x-\rho_1'(a')x)y'\\
  \nonumber&&+a'\circ (\rho_0(x)y')+\mu_0(\mu_0'(y')x)a',\\
  \label{1010} \rho_0'(x')(a\cdot y)&=&(\rho_0'(x')a-\mu_0'(x')a)\cdot y+\rho_1'(\mu_1(a)x'-\rho_1(a)x')y\\
  \nonumber&&+a\cdot (\rho_0'(x')y)+\mu_0'(\mu_0(y)x')a,\\
  \label{1111} \rho_1(a)(x'\circ y')&=&(\rho_1(a)x'-\mu_1(a)x')\circ y'+\rho_1(\mu_0'(x')a-\rho_0'(x')a)y'\\
 \nonumber&&+x'\circ (\rho_1(a)y')+\mu_1(\mu_0'(y')a)x',\\
 \label{1212} \rho_1'(a')(x\cdot y)&=&(\rho_1'(a')x-\mu_1'(a')x)\cdot y+\rho_1'(\mu_0(x)a'-\rho_0(x)a')y\\
 \nonumber&&+x\cdot (\rho_1'(a')y)+\mu_1'(\mu_0(y)a')x,\\
 \label{1313} \rho_0(x)(y'\circ z')&=&(\rho_0(x)y'-\mu_0(x)y')\circ z'+\rho_0(\mu_0'(y')x-\rho_0'(y')x)z'\\
 \nonumber&&+y'\circ (\rho_0(x)z')+\mu_0(\mu_0'(z')x)y',\\
 \label{1414} \rho_0'(x')(y\cdot z)&=&(\rho_0'(x')y-\mu_0'(x')y)\cdot z+\rho_0'(\mu_0(y)x'-\rho_0(y)x')z\\
 \nonumber&&+y\cdot (\rho_0'(x')z)+\mu_0'(\mu_0(z)x')y.
 \end{eqnarray}
 Then there exists a strict pre-Lie $2$-algebra $(A_0\oplus A_0', A_{-1}\oplus A_{-1}', \dM+\dM', \ast)$, where $\ast$ is given by
 \begin{eqnarray}
 (x+x')\ast(y+y')&=&x\cdot y+x'\circ y'+\rho_0(x)y'+\mu_0'(y')x+\rho_0'(x')y+\mu_0(y)x',\\
 (x+x')\ast(a+a')&=&x\cdot a+x'\circ a'+\rho_0(x)a'+\mu_1'(a')x+\rho_0'(x')a+\mu_1(a)x',\\
 (a+a')\ast(x+x')&=&a\cdot x+a'\circ x'+\rho_1(a)x'+\mu_0'(x')a+\rho_1'(a')x+\mu_0(x)a',
 \end{eqnarray}
 for any $x+x', y+y' \in A_0\oplus A_0'$, $a+a' \in A_{-1}\oplus A_{-1}'$.
 
 Conversely, given a strict pre-Lie $2$-algebra $(A_0\oplus A_0', A_{-1}\oplus A_{-1}', \dM+\dM', \ast)$, in which $\huaA=(A_0, A_{-1}, \dM, \cdot)$ and $\huaA'=(A_0', A_{-1}', \dM', \circ)$ are strict sub-pre-Lie $2$-algebras of $\huaA+\huaA'$. Then there exist strict representations $\rho=(\rho_0, \rho_1)$, $\mu=(\mu_0, \mu_1)$ of $\huaA$ on $\huaA'$ and $\rho'=(\rho_0', \rho_1')$, $\mu'=(\mu_0', \mu_1')$ of $\huaA'$ on $\huaA$ satisfying above fourteen equations such that $\ast$ is given.
 \end{thm}
 \begin{proof}
Firstly, we will show that ($a_1$)-($a_3$) hold in the definition of strict pre-Lie $2$-algebra.

For $x+x' \in A_0\oplus A_0'$, $a+a', b+b'\in A_{-1}\oplus A_{-1}'$, because $\mu=(\mu_0, \mu_1)$ is a chain map from $2$-term complex of vector spaces $\huaA$ to $\End(\huaA')$ and $\mu'=(\mu_0', \mu_1')$ is a chain map from $2$-term complex of vector spaces $\huaA'$ to $\End(\huaA)$, we have
\begin{eqnarray*}
	(\dM+\dM')((x+x')\ast(a+a'))&=&(\dM+\dM')(x\cdot a+x'\circ a'+\rho_0(x)a'+\mu_1'(a')x+\rho_0'(x')a+\mu_1(a)x')\\
	&=&\dM(x\cdot a+\mu_1'(a')x+\rho_0'(x')a)+\dM'(x'\circ a'+\rho_0(x)a'+\mu_1(a)x')\\
	&=&x\cdot{\dM a}+\mu_0'(\dM'a')x+\rho_0'(x')(\dM a)+x'\circ \dM'a'+\rho_0(x)(\dM'a')+\mu_0(\dM a)x'\\
	&=&(x+x')\ast(\dM+\dM')(a+a').
\end{eqnarray*}
Because $\rho=(\rho_0, \rho_1)$ is a strict representation of the strict Lie $2$-algebra $\huaG(\huaA)$ on $\huaV$ and $\rho'=(\rho_0', \rho_1')$ is a strict representation of the strict Lie $2$-algebra $\huaG(\huaA')$ on $\huaA$, we have
\begin{eqnarray*}
	(\dM+\dM')((a+a')\ast(x+x'))&=&(\dM+\dM')(a\cdot x+a'\circ x'+\rho_1(a)x'+\mu_0'(x')a+\rho_1'(a')x+\mu_0(x)a')\\
	&=&\dM(a\cdot x+\mu_0'(x')a+\rho_1'(a')x)+\dM'(a'\circ x'+\rho_1(a)x'+\mu_0(x)a')\\
	&=&(\dM a)\cdot x+(\dM'a')\circ x'+\mu_0'(x')(\dM a)+\rho_0'(\dM' a')x+\rho_0(\dM a)x'+\mu_0(x)(\dM' a')\\
	&=&((\dM+\dM')(a+a'))\ast(x+x').
\end{eqnarray*}
Similarly, we can get $$((\dM+\dM')(a+a'))\ast(b+b')=(a+a')\ast((\dM+\dM')(b+b')).$$

Secondly, we will show that ($b_1$)-($b_3$) hold in the definition of strict pre-Lie $2$-algebra.

For $x+x',y+y'$ and $z+z'$ in $A_0\oplus A_0'$, by \eqref{33}, \eqref{44}, \eqref{1313}, \eqref{1414}, \eqref{eq:rep1} of Definition \ref{Representations of strict pre-Lie 2-algebras} and Condition {\rm(ii)} of Definition \ref{strict homomorphism}, we have 
\begin{eqnarray*}
   &&((x+x')\ast(y+y'))\ast(z+z')-(x+x')\ast((y+y')\ast(z+z'))\\
   &&- ((y+y')\ast(x+x'))\ast(z+z')+(y+y')\ast((x+x')\ast(z+z'))\\
   &=&\Big((x\cdot y)\cdot z-x\cdot(y\cdot z)- (y\cdot x)\cdot z+y\cdot(x\cdot z)\Big)+\Big((x'\circ y')\circ z'-x'\circ (y'\circ z')-(y'\circ x')\circ z'\\
   &&+y'\circ (x'\circ z')\Big)+\Big(\mu_0(z)(\rho_0(x)y')-\rho_0(x)(\mu_0(z)y')-\mu_0(z)(\mu_0(x)y')+\mu_0(x\cdot z)y'\Big)\\
   &&+\Big(\mu_0(z)(\mu_0(y)x')-\mu_0(y\cdot z)x'-\mu_0(z)(\rho_0(y)x')+\rho_0(y)(\mu_0(z)x')\Big)+\Big(\mu_0'(z')(\mu_0'(y')x)\\
   &&-\mu_0'(y'\circ z')x-\mu_0'(z')(\rho_0'(y')x)+\rho_0'(y')(\mu_0'(z')x)\Big)+\Big(\mu_0'(z')(\rho_0'(x')y)-\rho_0'(x')(\mu_0'(z')y)\\
   &&-\mu_0'(z')(\mu_0'(x')y)+\mu_0'(x'\circ z')y\Big)+\Big((\mu_0'(y')x)\cdot z+\rho_0'(\rho_0(x)y')z-x\cdot(\rho_0'(y')z)-\mu_0'(\mu_0(z)y')x\\
   &&-(\rho_0'(y')x)\cdot z-\rho_0'(\mu_0(x)y')z+\rho_0'(y')(x\cdot z)\Big)+\Big((\rho_0'(x')y)\cdot z+\rho_0'(\mu_0(y)x')z-\rho_0'(x')(y\cdot z)\\
   &&-(\mu_0'(x')y)\cdot z-\rho_0'(\rho_0(y)x')z+y\cdot(\rho_0'(x')z)+\mu_0'(\mu_0(z)x')y\Big)+\Big((\rho_0(x)y')\circ z'+\rho_0(\mu_0'(y')x)z'\\
   &&-\rho_0(x)(y'\circ z')-(\mu_0(x)y')\circ z'-\rho_0(\rho_0'(y')x)z'+y'\circ (\rho_0(x)z')+\mu_0(\mu_0'(z')x)y'\Big)\\
   &&+\Big((\mu_0(y)x')\circ z'+\rho_0(\rho_0'(x')y)z'-x'\circ(\rho_0(y)z')-\mu_0(\mu_0'(z')y)x'-(\rho_0(y)x')\circ z'\\
   &&-\rho_0(\mu_0'(x')y)z'+\rho_0(y)(x'\circ z')\Big)+\Big(\mu_0'(z')(x\cdot y)-x\cdot(\mu_0'(z')y)-\mu_0'(\rho_0(y)z')x\\
   &&-\mu_0'(z')(y\cdot x)+y\cdot(\mu_0'(z')x)+\mu_0'(\rho_0(x)z')y\Big)+\Big(\mu_0(z)(x'\circ y')-x'\circ(\mu_0(z)y')\\
   &&-\mu_0(\rho_0'(y')z)x'-\mu_0(z)(y'\circ x')+y'\circ (\mu_0(z)x')+\mu_0(\rho_0'(x')z)y'\Big)+\Big(\rho_0(x\cdot y)z'\\
   &&-\rho_0(x)(\rho_0(y)z')-\rho_0(y\cdot x)z'+\rho_0(y)(\rho_0(x)z')\Big)+\Big(\rho_0'(x'\circ y')z-\rho_0'(x')(\rho_0'(y')z)\\
   &&-\rho_0'(y'\circ x')z+\rho_0'(y')(\rho_0'(x')z)\Big)\\
   &=&0,
\end{eqnarray*}
which implies that Condition {\rm ($b_1$)} holds in the definition of strict pre-Lie $2$-algebra.

For $x+x', y+y' \in A_0\oplus A_0'$, $a+a' \in A_{-1}\oplus A_{-1}'$, by \eqref{11}, \eqref{22}, \eqref{77}, \eqref{88}, \eqref{eq:rep2} of Definition \ref{Representations of strict pre-Lie 2-algebras} and Condition {\rm(ii)} of Definition \ref{strict homomorphism}, we have
\begin{eqnarray*}
	&&((x+x')\ast(y+y'))\ast(a+a')-(x+x')\ast((y+y')\ast(a+a'))\\
	&&-((y+y')\ast(x+x'))\ast(a+a')+(y+y')\ast((x+x')\ast(a+a'))\\
	&=&\Big((x\cdot y)\cdot a-x\cdot(y\cdot a)-(y\cdot x)\cdot a+y\cdot(x\cdot a)\Big)+\Big((x' \circ y')\circ a'-x'\circ (y' \circ a')-(y'\circ x')\circ a'\\
	&&+y'\circ(x'\circ a')\Big)+\Big(\mu_1(a)(\rho_0(x)y')-\rho_0(x)(\mu_1(a)y')-\mu_1(a)(\mu_0(x)y')+\mu_1(x\cdot a)y'\Big)\\
	&&+\Big(\mu_1(a)(\mu_0(y)x')-\mu_1(y\cdot a)x'-\mu_1(a)(\rho_0(y)x')+\rho_0(y)(\mu_1(a)x')\Big)+\Big(\mu_1'(a')(\mu_0'(y')x)\\
	&&-\mu_1'(y'\circ a')x-\mu_1'(a')(\rho_0'(y')x)+\rho_0'(y')(\mu_1'(a')x)\Big)+\Big(\mu_1'(a')(\rho_0'(x')y)-\rho_0'(x')(\mu_1'(a')y)\\
	&&-\mu_1'(a')(\mu_0'(x')y)+\mu_1'(x'\circ a')y\Big)+\Big((\rho_0(x)y')\circ a'+\rho_0(\mu_0'(y')x)a'-\rho_0(x)(y'\circ a')\\
	&&-(\mu_0(x)y')\circ a'-\rho_0(\rho_0'(y')x)a'+y'\circ(\rho_0(x)a')+\mu_1(\mu_1'(a')x)y'\Big)+\Big((\mu_0(y)x')\circ a'\\
	&&+\rho_0(\rho_0'(x')y)a'-x'\circ (\rho_0(y)a')-\mu_1(\mu_1'(a')y)x'-(\rho_0(y)x')\circ a'-\rho_0(\mu_0'(x')y)a'\\
	&&+\rho_0(y)(x'\circ a')\Big)+\Big((\mu_0'(y')x)\cdot a+\rho_0'(\rho_0(x)y')a-x\cdot(\rho_0'(y')a)-\mu_1'(\mu_1(a)y')x-(\rho_0'(y')x)\cdot a\\
	&&-\rho_0'(\mu_0(x)y')a+\rho_0'(y')(x\cdot a)\Big)+\Big((\rho_0'(x')y)\cdot a+\rho_0'(\mu_0(y)x')a-\rho_0'(x')(y\cdot a)-(\mu_0'(x')y)\cdot a\\
	&&-\rho_0'(\rho_0(y)x')a+y\cdot(\rho_0'(x')a)+\mu_1'(\mu_1(a)x')y\Big)+\Big(\mu_1'(a')(x\cdot y)-x\cdot(\mu_1'(a')y)-\mu_1'(\rho_0(y)a')x\\
	&&-\mu_1'(a')(y\cdot x)+y\cdot(\mu_1'(a')x)+\mu_1'(\rho_0(x)a')y\Big)+\Big(\mu_1(a)(x'\circ y')-x'\circ (\mu_1(a)y')\\
	&&-\mu_1(\rho_0'(y')a)x'-\mu_1(a)(y'\circ x')+y'\circ(\mu_1(a)x')+\mu_1(\rho_0'(x')a)y'\Big)+\Big(\rho_0(x\cdot y)a'\\
	&&-\rho_0(x)(\rho_0(y)a')-\rho_0(y\cdot x)a'+\rho_0(y)(\rho_0(x)a')\Big)+\Big(\rho_0'(x'\circ y')a-\rho_0'(x')(\rho_0'(y')a)\\
	&&-\rho_0'(y'\circ x')a+\rho_0'(y')(\rho_0'(x')a)\Big)\\
	&=&0,
\end{eqnarray*}
which implies that Condition {\rm ($b_2$)} holds in the definition of strict pre-Lie $2$-algebra. 
 
Similarly, we can get \begin{eqnarray*}
	&((x+x')\ast(a+a'))\ast(y+y')-(x+x')\ast((a+a')\ast(y+y'))\\
	&-((a+a')\ast(x+x'))\ast(y+y')+(a+a')\ast((x+x')\ast(y+y'))=0,
\end{eqnarray*}
which implies that Condition {\rm ($b_3$)} holds in the definition of strict pre-Lie $2$-algebra. 

Then $(A_0\oplus A_0', A_{-1}\oplus A_{-1}', \dM+\dM', \ast)$ is a strict pre-Lie $2$-algebra.
 \end{proof}

 \begin{defi}
 Let $\huaA=(A_0, A_{-1}, \dM, \cdot)$ and $\huaA'=(A_0', A_{-1}', \dM', \circ)$ be two strict pre-Lie $2$-algebras. Suppose that $(\rho_0, \rho_1)$, $(\mu_0, \mu_1): \huaA \longrightarrow \End(\huaA')$ and $(\rho_0', \rho_1')$, $(\mu_0', \mu_1'): \huaA' \longrightarrow \End(\huaA)$ are strict representations of $\huaA$ and $\huaA'$ on $\huaA'$ and $\huaA$, respectively. We call them a {\bf matched pair} if they satisfy \eqref{11}-\eqref{1414}. We denote it by $(\huaA, \huaA'; (\rho_0, \rho_1), (\mu_0, \mu_1), (\rho_0', \rho_1'), (\mu_0', \mu_1'))$.
 \end{defi}
 
In the following, we recall the notion of matched pair of strict Lie $2$-algebras. 
\begin{defi}\label{matched pair of strict Lie $2$-algebras.2}
 	Let $\g=(\g_0, \g_{-1}, \dM, [\cdot, \cdot])$ and $\g'=(\g_0', \g_{-1}', \dM', [\cdot, \cdot]')$ be two strict Lie $2$-algebras. Suppose that $(\mu_0, \mu_1): \g \longrightarrow \End(\g')$ and $(\mu_0', \mu_1'): \g' \longrightarrow \End(\g)$ are strict representations of $\g$ and $\g'$ on $\g'$ and $\g$, respectively. We call them a {\bf matched pair} and denote it by $(\g, \g'; (\mu_0, \mu_1), (\mu_0', \mu_1'))$ if they satisfy the following equations:
\begin{eqnarray}
 		\label{1} \mu_0'(x')[x, y]&=&[x, \mu_0'(x')y]+[\mu_0'(x')x, y]+\mu_0'(\mu_0(y)x')x-\mu_0'(\mu_0(x)x')y;\\
 		\label{2} \mu_0(x)[x', y']'&=&[x', \mu_0(x)y']'+[\mu_0(x)x', y']'+\mu_0(\mu_0'(y')x)x'-\mu_0(\mu_0'(x')x)y';\\
 		\label{3} \mu_1'(h')[x, y]&=&[x, \mu_1'(h')y]+[\mu_1'(h')x, y]+\mu_1'(\mu_0(y)h')x-\mu_1'(\mu_0(x)h')y;\\
 		\label{4} \mu_1(h)[x', y']'&=&[x', \mu_1(h)y']'+[\mu_1(h)x', y']'+\mu_1(\mu_0'(y')h)x'-\mu_1(\mu_0'(x')h)y';\\
 		\label{5} \mu_0'(x')[x, h]&=&[x, \mu_0'(x')h]+[\mu_0'(x')x, h]+\mu_1'(\mu_1(h)x')x-\mu_0'(\mu_0(x)x')h;\\
 		\label{6} \mu_0(x)[x', h']'&=&[x', \mu_0(x)h']'+[\mu_0(x)x', h']'+\mu_1(\mu_1'(h')x)x'-\mu_0(\mu_0'(x')x)h',
 	\end{eqnarray}
where $x,y\in\g_0$, $h\in\g_{-1}$, $x',y'\in\g'_{0}$, $h'\in\g'_{-1}$.
    
 \end{defi}

 \begin{thm}\label{matched pair of strict Lie $2$-algebras}
 Let $(\g, \g'; (\mu_0, \mu_1), (\mu_0', \mu_1'))$ be a matched pair of strict Lie $2$-algebras $\g$ and $\g'$. Then there exists a strict Lie $2$-algebra $(\g\oplus \g', \dM\oplus \dM', [\cdot, \cdot]_{\g\oplus \g'})$, where $[\cdot, \cdot]_{\g\oplus \g'}$ is given by
 	\begin{eqnarray}
 	\label{eq:Lie2mp-bracket1}	[x+x', y+y']_{\g\oplus \g'}&=&[x, y]+\mu_0(x)(y')-\mu_0'(y')x+\mu_0'(x')y-\mu_0(y)x'+[x', y']';\\\
 	\label{eq:Lie2mp-bracket2}	[x+x', h+h']_{\g\oplus \g'}&=&[x, h]+\mu_0(x)(h')-\mu_1'(h')x-\mu_1(h)x'+\mu_0'(x')h+[x', h']'.
 	\end{eqnarray}
 	
 	Conversely, given a strict Lie $2$-algebra $(\g\oplus \g', \dM\oplus \dM', [\cdot, \cdot]_{\g\oplus \g'})$, in which $\g$ and $\g'$ are strict sub-Lie $2$-algebras with respect to the restricted brackets, there exist representations $(\mu_0, \mu_1):\g\longrightarrow \End(\g')$ and $(\mu_0', \mu_1'):\g'\longrightarrow \End(\g)$ satisfying \eqref{1}-\eqref{6} such that the bracket $[\cdot, \cdot]_{\g\oplus \g'}$ is given by \eqref{eq:Lie2mp-bracket1} and \eqref{eq:Lie2mp-bracket2}.
 \end{thm}
 
 \begin{pro}
Let $(\huaA, \huaA'; (\rho_0,\rho_1), (\mu_0,\mu_1), (\rho'_0,\rho'_1), (\mu'_0,\mu'_1))$ be a matched pair of strict pre-Lie $2$-algebras $\huaA$ and $\huaA'$. Then $(\huaG(\huaA),\huaG(\huaA');(\rho_0-\mu_0,\rho_1-\mu_1),(\rho'_0-\mu'_0,\rho'_1-\mu'_1))$ is a matched pair of strict Lie $2$-algebras.
 \end{pro}
 \begin{proof}
 	Let $(\huaA, \huaA'; (\rho_0,\rho_1), (\mu_0,\mu_1), (\rho'_0,\rho'_1), (\mu'_0,\mu'_1))$ be a matched pair of strict pre-Lie $2$-algebras $\huaA$ and $\huaA'$. By \eqref{33} and \eqref{1414}, we have 
 	\begin{eqnarray*}
 	(\rho_0'-\mu_0')(x')[x,y]&=&(\rho_0'(x')x-\mu_0'(x')x)\cdot y+\rho_0'(\mu_0(x)x'-\rho_0(x)x')y+x\cdot(\rho_0'(x')y)\\
 	&&+\mu_0'(\mu_0(y)x')x-(\rho_0'(x')y-\mu_0'(x')y)\cdot x-\rho_0'(\mu_0(y)x'-\rho_0(y)x')x\\
 	&&-y\cdot(\rho_0'(x')x)-\mu_0'(\mu_0(x)x')y-x\cdot (\mu_0'(x')y)+y\cdot (\mu_0'(x')x)\\
 	&&-\mu_0'(\rho_0(y)x')x+\mu_0'(\rho_0(x)x')y\\
 	&=&[\rho_0'(x')x, y]-[\mu_0'(x')x, y]+[x, \rho_0'(x')y]-[x, \mu_0'(x')y]-\rho_0'((\rho_0-\mu_0)(x)x')y\\
 	&&+\rho_0'((\rho_0-\mu_0)(y)x')x-\mu_0'((\rho_0-\mu_0)(y)x')x+\mu_0'((\rho_0-\mu_0)(x)x')y\\
 	&=&[(\rho_0'-\mu_0')(x')x, y]+[x, (\rho_0'-\mu_0')(x')y]-(\rho_0'-\mu_0')((\rho_0-\mu_0)(x)x')y\\
 	&&+(\rho_0'-\mu_0')((\rho_0-\mu_0)(y)x')x,
 	\end{eqnarray*}
 which implies that \eqref{1} holds in Definition \ref{matched pair of strict Lie $2$-algebras.2}. 
 
 By \eqref{11} and \eqref{1212}, we have
 \begin{eqnarray*}
 (\rho_1'-\mu_1')(h')[x,y]&=&(\rho_1'(h')x-\mu_1'(h')x)\cdot y+\rho_1'(\mu_0(x)h'-\rho_0(x)h')y+x\cdot (\rho_1'(h')y)\\
 &&+\mu_1'(\mu_0(y)h')x-(\rho_1'(h')y-\mu_1'(h')y)\cdot x-\rho_1'(\mu_0(y)h'-\rho_0(y)h')x\\
 &&-y\cdot (\rho_1'(h')x)-\mu_1'(\mu_0(x)h')y-x\cdot (\mu_1'(h')y)+y\cdot (\mu_1'(h')x)\\
 &&-\mu_1'(\rho_0(y)h')x+\mu_1'(\rho_0(x)h')y\\
 &=&[\rho_1'(h')x, y]-[\mu_1'(h')x, y]+[x, \rho_1'(h')y]-[x, \mu_1'(h')y]-\rho_1'((\rho_0-\mu_0)(x)h')y\\
 &&+\rho_1'((\rho_0-\mu_0)(y)h')x-\mu_1'((\rho_0-\mu_0)(y)h')x+\mu_1'((\rho_0-\mu_0)(x)h')y\\
 &=&[(\rho_1'-\mu_1')(h')x, y]+[x, (\rho_1'-\mu_1')(h')y]-(\rho_1'-\mu_1')((\rho_0-\mu_0)(x)h')y\\
 &&+(\rho_1'-\mu_1')((\rho_0-\mu_0)(y)h')x,
 \end{eqnarray*}
  which implies that \eqref{3} holds in Definition \ref{matched pair of strict Lie $2$-algebras.2}. 
  
  By \eqref{55}, \eqref{88} and \eqref{1010}, we have
  \begin{eqnarray*}
  (\rho_0'-\mu_0')(x')[x,h]&=&(\rho_0'(x')x-\mu_0'(x')x)\cdot h+\rho_0'(\mu_0(x)x'-\rho_0(x)x')h+x\cdot (\rho_0'(x')h)\\
  &&+\mu_1'(\mu_1(h)x')x-(\rho_0'(x')h-\mu_0'(x')h)\cdot x-\rho_1'(\mu_1(h)x'-\rho_1(h)x')x\\
  &&-h\cdot (\rho_0'(x')x)-\mu_0'(\mu_0(x)x')h-x\cdot (\mu_0'(x')h)+h\cdot (\mu_0'(x')x)\\
  &&-\mu_1'(\rho_1(h)x')x+\mu_0'(\rho_0(x)x')h\\
  &=&[\rho_0'(x')x, h]-[\mu_0'(x')x, h]+[x, \rho_0'(x')h]-[x, \mu_0'(x')h]-\rho_0'((\rho_0-\mu_0)(x)x')h\\
  &&+\rho_1'((\rho_1-\mu_1)(h)x')x-\mu_1'((\rho_1-\mu_1)(h)x')x+\mu_0'((\rho_0-\mu_0)(x)x')h\\
  &=&[(\rho_0'-\mu_0')(x')x, h]+[x, (\rho_0'-\mu_0')(x')h]-(\rho_0'-\mu_0')((\rho_0-\mu_0)(x)x')h\\
  &&+(\rho_1'-\mu_1')((\rho_1-\mu_1)(h)x')x,
  \end{eqnarray*}
  which implies that \eqref{5} holds in Definition \ref{matched pair of strict Lie $2$-algebras.2}. 
  
  Similarly, we can get \eqref{2}, \eqref{4} and \eqref{6} of Definition \ref{matched pair of strict Lie $2$-algebras.2}. Then $(\huaG(\huaA),\huaG(\huaA');(\rho_0-\mu_0,\rho_1-\mu_1),(\rho'_0-\mu'_0,\rho'_1-\mu'_1))$ is a matched pair of strict Lie $2$-algebras.
 \end{proof}

\begin{thm}
Let $\huaA=(A_0, A_{-1}, \dM, \cdot)$ be a strict pre-Lie $2$-algebra. Suppose that there is another strict pre-Lie $2$-algebra structure on its dual space $\huaA^*$. Then $(\huaG(\huaA), \huaG^*(\huaA^*); (L_0^*, L_1^*), (\huaL_0^*, \huaL_1^*))$ is a matched pair of Lie $2$-algebras if and only if $(\huaA, \huaA^*; (\ad_0^*, \ad_1^*), (-R_0^*, -R_1^*), (\mathfrak{ad}_0^*, \mathfrak{ad}_1^*), (-\huaR_0^*, -\huaR_1^*))$ is a matched pair of pre-Lie $2$-algebras. 
\end{thm} 
\begin{proof}
For any $x, y ,z \in A_0$, $a, b, c \in A_{-1}$, $x^*, y^* ,z^* \in A_0^*$, $a^*, b^*, c^* \in A_{-1}^*$, we have
\begin{eqnarray*}
{\rm Eq.~}(\ref{1}) &\iff& {\rm Eq.~}(\ref{11}) \iff {\rm Eq.~}(\ref{77})\\
{\rm Eq.~}(\ref{2}) &\iff& {\rm Eq.~}(\ref{22}) \iff {\rm Eq.~}(\ref{88})\\
{\rm Eq.~}(\ref{3}) &\iff& {\rm Eq.~}(\ref{33}) \iff {\rm Eq.~}(\ref{99})\\
{\rm Eq.~}(\ref{4}) &\iff& {\rm Eq.~}(\ref{44}) \iff {\rm Eq.~}(\ref{1010})\\
{\rm Eq.~}(\ref{5}) &\iff& {\rm Eq.~}(\ref{55}) \iff {\rm Eq.~}(\ref{1111}) \iff {\rm Eq.~}(\ref{1313})\\
{\rm Eq.~}(\ref{6}) &\iff& {\rm Eq.~}(\ref{66}) \iff {\rm Eq.~}(\ref{1212}) \iff{\rm Eq.~}(\ref{1414}).
\end{eqnarray*}
As an example, we show how \eqref{1} is equivalent to \eqref{11} and \eqref{77}. In fact, it follows from
\begin{eqnarray*}
\left\langle \huaL_0^*(a^*)[x, y], z^*\right\rangle&=&\left\langle [x, y], -a^*\circ z^*\right\rangle=\left\langle \huaR_1^*(z^*)[x, y], a^*\right\rangle;\\
\left\langle [x, \huaL_0^*(a^*)y], z^*\right\rangle&=&\left\langle \huaL_0^*(a^*)y, -ad_0^*(x)z^*\right\rangle=\left\langle y, a^*\circ (ad_0^*(x)z^*)\right\rangle=\left\langle -\huaR_1^*(ad_0^*(x)z^*)y, a^*\right\rangle;\\
\left\langle [\huaL_0^*(a^*)x, y], z^*\right\rangle&=&\left\langle -[y, \huaL_0^*(a^*)x], z^*\right\rangle=\left\langle \huaL_0^*(a^*)x, ad_0^*(y)z^*\right\rangle=\left\langle x, -a^*\circ (ad_0^*(y)z^*)\right\rangle\\
&=&\left\langle \huaR_1^*(ad_0^*(y)z^*)x, a^*\right\rangle;\\
\left\langle \huaL_0^*(L_0^*(y)a^*)x, z^*\right\rangle&=&\left\langle x, -(L_0^*(y)a^*)\circ z^*\right\rangle=\left\langle \huaR_1^*(z^*)x, L_0^*(y)a^*\right\rangle=\left\langle -y\cdot(\huaR_1^*(z^*)x), a^*\right\rangle;\\
-\left\langle \huaL_0^*(L_0^*(x)a^*)y, z^*\right\rangle&=&\left\langle y, (L_0^*(x)a^*)\circ z^*\right\rangle=\left\langle \huaR_1^*(z^*)y, -L_0^*(x)a^*\right\rangle=\left\langle x\cdot(\huaR_1^*(z^*)y), a^*\right\rangle,
\end{eqnarray*}
which impliy that \eqref{1} $\iff$ \eqref{11}.

Furthermore, we can also get
\begin{eqnarray*}
	\left\langle \huaL_0^*(a^*)[x, y], z^*\right\rangle&=&\left\langle [x, y], -a^*\circ z^*\right\rangle=\left\langle y, ad_0^*(x)(a^*\circ z^*)\right\rangle;\\
	\left\langle [x, \huaL_0^*(a^*)y], z^*\right\rangle&=&\left\langle \huaL_0^*(a^*)y, -ad_0^*(x)z^*\right\rangle=\left\langle y, a^*\circ (ad_0^*(x)z^*)\right\rangle;\\
	\left\langle [\huaL_0^*(a^*)x, y], z^*\right\rangle&=&\left\langle y, -ad_0^*(\huaL_0^*(a^*)x)z^*\right\rangle=\left\langle y, -ad_0^*(\mathfrak{ad}_0^*(a^*)x+\huaR_0^*(a^*)x)z^*\right\rangle;\\
	\left\langle \huaL_0^*(L_0^*(y)a^*)x, z^*\right\rangle&=&\left\langle x, -(L_0^*(y)a^*)\circ z^*\right\rangle=\left\langle \huaR_1^*(z^*)x, L_0^*(y)a^*\right\rangle=\left\langle -y\cdot(\huaR_1^*(z^*)x), a^*\right\rangle\\
	&=&\left\langle y, R_1^*(\huaR_1^*(z^*)x)a^*\right\rangle;\\
	-\left\langle \huaL_0^*(L_0^*(x)a^*)y, z^*\right\rangle&=&\left\langle y, (L_0^*(x)a^*)\circ z^*\right\rangle=\left\langle y, (ad_0^*(x)a^*+R_0^*(x)a^*)\circ z^*\right\rangle,
\end{eqnarray*}
which implies that \eqref{1} $\iff$ \eqref{77}. And the proof of other equivalence is similar. The conclusion follows.
\end{proof}
 
\section{Strict pre-Lie 2-bialgebras and their constructions}\label{sec:Pre-Lie 2-bialgebras}
In this section, we introduce the notion of a strict pre-Lie 2-bialgebra and show that there is an one-to-one correspondence between  strict pre-Lie 2-bialgebras and matched pairs of strict pre-Lie 2-algebras. Then by the cohomology of strict Lie 2-algebras, we introduce the coboundary strict pre-Lie 2-bialgebra, which leads to the 2-graded classical Yang-Baxter Equations on strict pre-Lie 2-algebras. We give an operator form description of the 2-graded classical Yang-Baxter Equations. Finally, we use the $\huaO$-operators on strict Lie 2-algebras to construct solutions of 2-graded classical Yang-Baxter Equations on a bigger strict pre-Lie 2-algebra.
\subsection{Strict pre-Lie 2-bialgebras}
In order to give the definition of strict pre-Lie $2$-bialgebras, we first recall the coboundary operator associated to a certain tensor representation associated to the strict pre-Lie $2$-algebra.  Let $\huaA$ be a  strict pre-Lie $2$-algebra. Then $\huaG(\huaA)$ is strict Lie $2$-algebra and $(\huaA\otimes \huaA;(L_0\otimes 1+1\otimes \ad_0, L_1\otimes 1+1\otimes \ad_1))$ is a representation of $\huaG(\huaA)$. Then $\huaA$ acts on a $3$-term complex of vector spaces $\huaA\otimes\huaA$ with 
\begin{eqnarray*}
(\huaA\otimes\huaA)_{0}:=\huaA_{-1}\otimes \huaA_{-1}\stackrel{\dM^{\otimes}}{\longrightarrow}(\huaA\otimes\huaA)_1:=A_0\otimes A_{-1}\oplus A_{-1}\otimes A_0\stackrel{\dM^{\otimes}}{\longrightarrow}(\huaA\otimes\huaA)_2:=A_0\otimes A_0,
\end{eqnarray*}
where  $\dM^{\otimes}$ is given by 
\begin{eqnarray*}
\dM^{\otimes}(a\otimes b)&=&(\dM\otimes 1+1\otimes \dM)(a\otimes b)=\dM a\otimes b+a\otimes \dM b, \quad a,b\in A_{-1},\\
\dM^{\otimes}(x\otimes a+b\otimes y)&=&(\dM\otimes 1-1\otimes \dM)(x\otimes a+b\otimes y)=\dM a\otimes y-x\otimes \dM b, \quad x,y\in A_0, a,b\in A_{-1}.
\end{eqnarray*}
The Chevalley-Eilenberg complex is given by 
\begin{eqnarray*}
&&(\huaA\otimes\huaA)_{0}\stackrel{D}{\longrightarrow}(\huaA\otimes\huaA)_1\oplus \Hom(A_0, (\huaA\otimes\huaA)_0)\stackrel{D}{\longrightarrow}\\
&&(\huaA\otimes\huaA)_2\oplus \Hom(A_0, (\huaA\otimes\huaA)_1)\oplus \Hom(A_{-1}, (\huaA\otimes\huaA)_{0})\oplus \Hom(\land^2 A_0, (\huaA\otimes\huaA)_0)\stackrel{D}{\longrightarrow}\\
&&\Hom(A_0, (\huaA\otimes\huaA)_2)\oplus \Hom(A_{-1}, (\huaA\otimes\huaA)_1)\oplus \Hom(\land^2 A_0, (\huaA\otimes\huaA)_1)\oplus \Hom(\land^3 A_0, (\huaA\otimes\huaA)_0)\\
&&\oplus \Hom(A_0\otimes A_{-1}, (\huaA\otimes\huaA)_0)\stackrel{D}{\longrightarrow}\cdots.
\end{eqnarray*}
The differential operator $D=\hat{\dM}+\bar{\dM}+\hat{\dM^{\otimes}}$, in which $\bar{\dM}$ is the operator associated to the tensor representation $(\ad_0\otimes 1+1\otimes L_0, \ad_1\otimes 1+1\otimes L_1)$ of $\huaA$ on $\huaA\otimes\huaA$. For a $1$-cochain $(\alpha_0, \alpha_1) \in \Hom(A_0, (\huaA\otimes\huaA)_1)\oplus \Hom(A_{-1}, (\huaA\otimes\huaA)_0)$, we have 
$$D(\alpha_0, \alpha_1)=-\dM^{\otimes}\circ\alpha_0+\bar{\dM}\alpha_0-\alpha_0\circ \dM+\dM^{\otimes}\circ\alpha_1+\bar{\dM}\alpha_1.$$
This implies that $(\alpha_0, \alpha_1)$ is a $1$-cocycle if and only if the following equations hold:
\begin{eqnarray}
(\dM\otimes 1-1\otimes \dM)\circ\alpha_0&=&0, \quad \alpha_0\circ \dM-(\dM\otimes 1+1\otimes \dM)\circ \alpha_1=0,\\
\bar{\dM}\alpha_0(x, y)&=&0, \quad \bar{\dM}\alpha_0(x, h)+\bar{\dM}\alpha_1(x, h)=0.
\end{eqnarray}

\begin{defi}\label{a strict pre-Lie 2-bialgebra}
A {\bf strict pre-Lie 2-bialgebra} $(\huaA,\huaA^*)$ consists of the following data:
\begin{itemize}
    \item[{\rm (i)}] $\huaA=(A_0,A_{-1},\dM,\cdot)$ is a strict pre-Lie $2$-algebra, which induces linear maps $\beta_1:A_0^*:\longrightarrow A_0^* \otimes A_0^*$ and $\beta_{0}:A_{-1}^*:\longrightarrow A_0^* \otimes A_{-1}^*\oplus A_{-1}^*\otimes A_0^*$ given by
    \begin{eqnarray*}
   \langle\beta_1(x^*),x\otimes y\rangle&=&\langle x^*,x\cdot y\rangle,\quad \langle\beta_{0}(a^*),x\otimes a\rangle=\langle a^*,x\cdot a\rangle,\\ 
    \langle\beta_{0}(a^*),a\otimes x\rangle&=&\langle a^*,a\cdot x\rangle,\quad x,y\in A_0,x^*\in A^*_0,a\in A_{-1},a^*\in A^*_{-1};
    \end{eqnarray*}
     \item[{\rm (ii)}] $\huaA^*=(A^*_{-1},A^*_{0},\dM^*,\ast)$ is a strict pre-Lie $2$-algebra, which induces linear maps $\alpha_{1}:A_{-1}:\longrightarrow A_{-1} \otimes A_{-1}$ and $\alpha_{0}:A_{0}:\longrightarrow A_{-1}\otimes A_{0}\oplus A_{0}\otimes A_{-1}$ given by
    \begin{eqnarray*}
   \langle\alpha_{1}(a),b^*\otimes c^*\rangle&=&\langle a,b^*\ast c^*\rangle,\quad \langle\alpha_{0}(x),a^*\otimes x^*\rangle=\langle x,a^*\ast x^*\rangle,\\ 
    \langle\alpha_{0}(x),x^*\otimes a^*\rangle&=&\langle x,x^*\ast a^*\rangle,\quad x\in A_0,x^*\in A^*_0,a\in A_{-1},a^*,b^*,c^*\in A^*_{-1};
    \end{eqnarray*}
     \item[{\rm (iii)}]$(\alpha_0, \alpha_1)$ is a $1$-cocycle of $\huaG(\huaA)$ associated to the representation $(L_0\otimes 1+1\otimes \ad_0, L_1\otimes 1+1\otimes \ad_1)$ with values in $\huaA\otimes\huaA$;
     \item[{\rm (iv)}]$(\beta_0, \beta_1)$ is a $1$-cocycle of $\huaG(\huaA^*)$ associated to the representation  $(\huaL_0\otimes 1+1\otimes \mathfrak{ad}_0, \huaL_1\otimes 1+1\otimes \mathfrak{ad}_1)$ with values in $\huaA^*\otimes\huaA^*$.
\end{itemize}
\end{defi}

\begin{pro}
  Let $\huaA=(A_0, A_{-1}, \dM, \cdot)$ be a strict pre-Lie $2$-algebra and $\huaA^*=(A^*_{-1},A^*_{0},\dM^*,\ast)$ be a strict pre-Lie $2$-algebra on the dual space of $\huaA$. And $\alpha_0, \alpha_1, \beta_0, \beta_1$ are given by Definition \ref{a strict pre-Lie 2-bialgebra}. Then $(\huaG(\huaA), \huaG(\huaA^*); (L_0^*, L_1^*),(\huaL_0^*, \huaL_1^*))$ is a matched pair of strict Lie $2$-algebras if and only if $(\huaA, \huaA^*)$ is a strict pre-Lie $2$-bialgebra. 
\end{pro}
\begin{proof}
There is a $3$-term complex of vector spaces $$A_{-1}\otimes A_{-1}\stackrel{\partial}{\longrightarrow}A_{-1}\otimes A_0\oplus A_0\otimes A_{-1}\stackrel{\partial}{\longrightarrow}A_0\otimes A_0.$$
If we want to prove that $(\alpha_0, \alpha_1)$ is a $1$-cocycle, that is $D(\alpha_0, \alpha_1)=0$, where 
\begin{eqnarray*}
D(\alpha_0, \alpha_1)=(\hat{\dM}+\bar{\dM}+\hat{\partial})(\alpha_0, \alpha_1)=-\alpha_0\circ \dM+\bar{\dM}\alpha_0+\bar{\dM}\alpha_1-\partial\circ\alpha_0+\partial\circ\alpha_1=0.
\end{eqnarray*}
This implies that $(\alpha_0, \alpha_1)$ is a $1$-cocycle if and only if the following equations hold:
\begin{itemize}
\item[\rm(i)] \quad $-(\dM\otimes 1-1\otimes \dM)\circ\alpha_0(x)=0$;
\item[\rm(ii)] \quad $-\alpha_0\circ \dM(a)+(\dM\otimes 1+1\otimes \dM)\circ\alpha_1(a)=0$;
\item[\rm(iii)] \quad $\bar{\dM}\alpha_0(x, y)=0$;
\item[\rm(iv)] \quad $\bar{\dM}\alpha_0(x, a)+\bar{\dM}\alpha_1(x, a)=0$.
\end{itemize}

Because the strict pre-Lie $2$-algebra structure on $\huaA^*$ is given by linear maps $\alpha_1^*:A_{-1}^*\otimes A_{-1}^* \longrightarrow A_{-1}^*$, $\alpha_0^*:A_{-1}^*\otimes A_0^*\oplus A_0^*\otimes A_{-1}^* \longrightarrow A_0^*$, we can get  
\begin{eqnarray*}
\left\langle (\dM\otimes 1-1\otimes \dM)\circ\alpha_0(x), y^*\otimes z^*\right\rangle&=&\left\langle \alpha_0(x), (\dM^*\otimes 1-1\otimes \dM^*)(y^*\otimes z^*)\right\rangle\\
&=&\left\langle x, \alpha_0^*((\dM^*y^*)\otimes z^*-y^*\otimes (\dM^*z^*))\right\rangle\\
&=&\left\langle x, (\dM^*y^*)\circ z^*-y^*\circ (\dM^*z^*)\right\rangle\\
&=&0,
\end{eqnarray*}
which implies that $(\dM\otimes 1-1\otimes \dM)\circ\alpha_0(x)=0$.
\begin{eqnarray*}
\left\langle -\alpha_0\circ \dM(a), b^*\otimes x^*+y^*\otimes c^*\right\rangle&=&\left\langle \dM(a), -\alpha_0^*(b^*\otimes x^*+y^*\otimes c^*)\right\rangle\\
&=&\left\langle \dM(a), -(b^*\circ x^*+y^*\circ c^*)\right\rangle\\
&=&\left\langle a, -\dM^*(b^*\circ x^*+y^*\circ c^*)\right\rangle\\
&=&\left\langle a, -(b^*\circ (\dM^*x^*)+(\dM^*y^*)\circ c^*)\right\rangle,
\end{eqnarray*}
\begin{eqnarray*}
\left\langle (\dM\otimes 1+1\otimes \dM)\circ\alpha_1(a), b^*\otimes x^*+y^*\otimes c^*\right\rangle&=&\left\langle \alpha_1(a), (\dM^*\otimes 1+1\otimes \dM^*)(b^*\otimes x^*+y^*\otimes c^*)\right\rangle\\
&=&\left\langle \alpha_1(a), (\dM^*y^*)\otimes c^*+b^*\otimes (\dM^*x^*)\right\rangle\\
&=&\left\langle a, \alpha_1^*((\dM^*y^*)\otimes c^*+b^*\otimes (\dM^*x^*))\right\rangle\\
&=&\left\langle a, ((\dM^*y^*)\circ c^*+b^*\circ (\dM^*x^*))\right\rangle.
\end{eqnarray*}
which implies that $-\alpha_0\circ \dM(a)+(\dM\otimes 1+1\otimes \dM)\circ\alpha_1(a)=0$.
So we can know that equation {\rm(i)} and equation {\rm(ii)} naturally hold.

In the case, equation {\rm(iii)} is 
$$\bar{\dM}\alpha_0(x, y)=(L_0\otimes 1+1\otimes ad_0)(x)\alpha_0(y)-(L_0\otimes 1+1\otimes ad_0)(y)\alpha_0(x)-\alpha_0([x, y]).$$ 
By calculation, we have 
\begin{eqnarray*}
\left\langle \huaL_0^*(a^*)[x, y], z^*\right\rangle&=&\left\langle [x, y], -a^*\circ z^*\right\rangle=\left\langle [x, y], -\alpha_0^*(a^*\otimes z^*)\right\rangle=\left\langle -\alpha_0([x, y]), a^*\otimes z^*\right\rangle,\\
-\left\langle [x, \huaL_0^*(a^*)y], z^*\right\rangle&=&-\left\langle \huaL_0^*(a^*)y, -\ad_0^*(x)z^*\right\rangle=-\left\langle y, a^*\circ(\ad_0^*(x)z^*)\right\rangle\\
&=&-\left\langle y, \alpha_0^*(a^*\otimes(\ad_0^*(x)z^*))\right\rangle=-\left\langle \alpha_0(y), a^*\otimes(\ad_0^*(x)z^*)\right\rangle\\
&=&-\left\langle \alpha_0(y), (1\otimes \ad_0^*(x))(a^*\otimes z^*)\right\rangle=\left\langle (1\otimes \ad_0(x))\alpha_0(y), a^*\otimes z^*\right\rangle,\\
-\left\langle[\huaL_0^*(a^*)x, y], z^*\right\rangle&=&\left\langle [y, \huaL_0^*(a^*)x], z^*\right\rangle=\left\langle \huaL_0^*(a^*)x, -\ad_0^*(y)z^*\right\rangle\\
&=&\left\langle x, a^*\circ(\ad_0^*(y)z^*)\right\rangle=\left\langle x, \alpha_0^*(a^*\otimes(\ad_0^*(y)z^*))\right\rangle\\
&=&\left\langle \alpha_0(x), (1\otimes \ad_0^*(y))(a^*\otimes z^*)\right\rangle=\left\langle -(1\otimes \ad_0(y))\alpha_0(x), (a^*\otimes z^*)\right\rangle,\\
-\left\langle \huaL_0^*(L_0^*(y)a^*)x, z^*\right\rangle&=&\left\langle x, (L_0^*(y)a^*)\circ z^*\right\rangle=\left\langle x, \alpha_0^*((L_0^*(y)a^*)\otimes z^*)\right\rangle\\
&=&\left\langle \alpha_0(x), (L_0^*(y)a^*)\otimes z^*\right\rangle=\left\langle \alpha_0(x), (L_0^*(y)\otimes 1)(a^*\otimes z^*)\right\rangle\\
&=&\left\langle -(L_0(y)\otimes 1)\alpha_0(x), a^*\otimes z^*\right\rangle,\\
\left\langle \huaL_0^*(L_0^*(x)a^*)y, z^*\right\rangle&=&\left\langle y, -(L_0^*(x)a^*)\circ z^*\right\rangle=\left\langle y, -\alpha_0^*((L_0^*(x)a^*)\otimes z^*)\right\rangle\\
&=&\left\langle \alpha_0(y), -(L_0^*(x)a^*)\otimes z^*\right\rangle=\left\langle \alpha_0(y), -(L_0^*(x)\otimes 1)(a^*\otimes z^*)\right\rangle\\
&=&\left\langle (L_0(x)\otimes 1)\alpha_0(y), a^*\otimes z^*\right\rangle,\\
\end{eqnarray*}
which implies that 
\begin{eqnarray*}
&&\left\langle \huaL_0^*(a^*)[x, y]-[x, \huaL_0^*(a^*)y]-[\huaL_0^*(a^*)x, y]-\huaL_0^*(L_0^*(y)a^*)x+\huaL_0^*(L_0^*(x)a^*)y, z^*\right\rangle\\
&=&\left\langle (L_0\otimes 1+1\otimes \ad_0)(x)\alpha_0(y)-(L_0\otimes 1+1\otimes \ad_0)(y)\alpha_0(x)-\alpha_0([x, y]), a^*\otimes z^*\right\rangle.
\end{eqnarray*}
So we can get the equation (iii) is equal to \eqref{1}.

Furthermore, we also have 
\begin{eqnarray*}
\left\langle \huaL_1^*(z^*)[x, y], a^*\right\rangle&=&\left\langle [x, y], -z^*\circ a^*\right\rangle=\left\langle [x, y], -\alpha_0^*(z^*\otimes a^*)\right\rangle=\left\langle -\alpha_0([x, y]), z^*\otimes a^*\right\rangle\\
-\left\langle [x, \huaL_1^*(z^*)y], a^*\right\rangle&=&\left\langle  \huaL_1^*(z^*)y, \ad_0^*(x)a^*\right\rangle=\left\langle y, -z^*\circ (\ad_0^*(x)a^*)\right\rangle\\
&=&\left\langle y, -\alpha_0^*(z^*\otimes (\ad_0^*(x)a^*))\right\rangle=\left\langle \alpha_0(y), -(z^*\otimes (\ad_0^*(x)a^*))\right\rangle\\
&=&\left\langle \alpha_0(y), -(1\otimes \ad_0^*(x))(z^*\otimes a^*)\right\rangle=\left\langle (1\otimes \ad_0(x))\alpha_0(y), z^*\otimes a^*\right\rangle\\
-\left\langle [\huaL_1^*(z^*)x, y], a^*\right\rangle&=&\left\langle [y, \huaL_1^*(z^*)x], a^*\right\rangle=\left\langle \huaL_1^*(z^*)x, -\ad_0^*(y)a^*\right\rangle\\
&=&\left\langle x, z^*\circ (\ad_0^*(y)a^*)\right\rangle=\left\langle x, \alpha_0^*(z^*\otimes (\ad_0^*(y)a^*))\right\rangle\\
&=&\left\langle \alpha_0(x), (1\otimes \ad_0^*(y))(z^*\otimes a^*)\right\rangle=\left\langle -(1\otimes \ad_0(y))\alpha_0(x), z^*\otimes a^*\right\rangle\\
-\left\langle \huaL_1^*(L_0^*(y)z^*)x, a^*\right\rangle&=&\left\langle x, (L_0^*(y)z^*)\circ a^*\right\rangle=\left\langle x, \alpha_0^*((L_0^*(y)z^*)\otimes a^*)\right\rangle\\
&=&\left\langle \alpha_0(x), (L_0^*(y)\otimes 1)(z^*\otimes a^*)\right\rangle=\left\langle -(L_0(y)\otimes 1)\alpha_0(x), z^*\otimes a^*\right\rangle\\
\left\langle \huaL_1^*(L_0^*(x)z^*)y, a^*\right\rangle&=&\left\langle y, -(L_0^*(x)z^*)\circ a^*\right\rangle=\left\langle y, -\alpha_0^*((L_0^*(x)z^*)\otimes a^*)\right\rangle\\
&=&\left\langle \alpha_0(y), -(L_0^*(x)\otimes 1)(z^*\otimes a^*)\right\rangle=\left\langle (L_0(x)\otimes 1)\alpha_0(y), z^*\otimes a^*\right\rangle,
\end{eqnarray*}
which implies that 
\begin{eqnarray*}
	&&\left\langle \huaL_1^*(z^*)[x, y]-[x, \huaL_1^*(z^*)y]-[\huaL_1^*(z^*)x, y]-\huaL_1^*(L_0^*(y)z^*)x+\huaL_1^*(L_0^*(x)z^*)y, a^*\right\rangle\\
	&=&\left\langle (L_0\otimes 1+1\otimes \ad_0)(x)\alpha_0(y)-(L_0\otimes 1+1\otimes \ad_0)(y)\alpha_0(x)-\alpha_0([x, y]), z^*\otimes a^*\right\rangle.
\end{eqnarray*}
So we can get the equation {\rm(iii)} is also equal to \eqref{3}.

Finally, equation {\rm(iv)} is 
$$\bar{\dM}\alpha_0(x, a)+\bar{\dM}\alpha_1(x, a)=-(L_1\otimes 1+1\otimes \ad_1)(a)\alpha_0(x)+(L_0\otimes 1+1\otimes \ad_0)(x)\alpha_1(a)-\alpha_1([x, a]).$$
We have
\begin{eqnarray*}
\left\langle \huaL_0^*(a^*)[x,c], b^*\right\rangle&=&\left\langle [x,c], -a^*\circ b^*\right\rangle=\left\langle [x,c], -\alpha_1^*(a^*\otimes b^*)\right\rangle=\left\langle -\alpha_1([x,c]), a^*\otimes b^*\right\rangle,\\
-\left\langle [x, \huaL_0^*(a^*)c], b^*\right\rangle&=&\left\langle \huaL_0^*(a^*)c, \ad_0^*(x)b^*\right\rangle=\left\langle c, -a^*\circ (\ad_0^*(x)b^*)\right\rangle\\
&=&\left\langle c, -\alpha_1^*(a^*\otimes (\ad_0^*(x)b^*))\right\rangle=\left\langle \alpha_1(c), -(1\otimes \ad_0^*(x))(a^*\otimes b^*)\right\rangle\\
&=&\left\langle (1\otimes \ad_0(x))\alpha_1(c), a^*\otimes b^*\right\rangle,\\
-\left\langle [\huaL_0^*(a^*)x, c], b^*\right\rangle&=&\left\langle [c, \huaL_0^*(a^*)x], b^*\right\rangle=\left\langle \huaL_0^*(a^*)x, -\ad_1^*(c)b^*\right\rangle\\
&=&\left\langle x, a^*\circ (\ad_1^*(c)b^*)\right\rangle=\left\langle x, \alpha_0^*(a^*\otimes (\ad_1^*(c)b^*))\right\rangle\\
&=&\left\langle \alpha_0(x), (1\otimes \ad_1^*(c))(a^*\otimes b^*)\right\rangle=\left\langle -(1\otimes \ad_1(c))\alpha_0(x), a^*\otimes b^*\right\rangle,\\
-\left\langle \huaL_1^*(L_1^*(c)a^*)x, b^*\right\rangle&=&\left\langle x, (L_1^*(c)a^*)\circ b^*\right\rangle=\left\langle x, \alpha_0^*((L_1^*(c)a^*)\otimes b^*)\right\rangle\\
&=&\left\langle \alpha_0(x), (L_1^*(c)\otimes 1)(a^*\otimes b^*)\right\rangle=\left\langle -(L_1(c)\otimes 1)\alpha_0(x), a^*\otimes b^*\right\rangle,\\
\left\langle \huaL_0^*(L_0^*(x)a^*)c, b^*\right\rangle&=&\left\langle c, -(L_0^*(x)a^*)\circ b^*\right\rangle=\left\langle c, -\alpha_1^*((L_0^*(x)a^*)\otimes b^*)\right\rangle\\
&=&\left\langle \alpha_1(c), -(L_0^*(x)\otimes 1)(a^*\otimes b^*)\right\rangle=\left\langle (L_0(x)\otimes 1)\alpha_1(c), a^*\otimes b^*\right\rangle.
\end{eqnarray*}
which implies that 
\begin{eqnarray*}
	&&\left\langle \huaL_0^*(a^*)[x,c]-[x, \huaL_0^*(a^*)c]-[\huaL_0^*(a^*)x, c]-\huaL_1^*(L_1^*(c)a^*)x+\huaL_0^*(L_0^*(x)a^*)c, b^*\right\rangle\\
	&=&\left\langle -(L_1\otimes 1+1\otimes \ad_1)(a)\alpha_0(x)+(L_0\otimes 1+1\otimes \ad_0)(x)\alpha_1(a)-\alpha_1([x, a]), a^*\otimes b^*\right\rangle.
\end{eqnarray*}
So we can get the equation {\rm(iv)} is equal to \eqref{5}. By the above calculation, we can get that $D(\alpha_0, \alpha_1)=0$ if and only if \eqref{1}, \eqref{3} and \eqref{5} of Definition \ref{matched pair of strict Lie $2$-algebras.2} are satisfied. 

Similarly, we can get that $D(\beta_0, \beta_1)=0$ if and only if \eqref{2}, \eqref{4} and \eqref{6} of Definition \ref{matched pair of strict Lie $2$-algebras.2} are satisfied. Then $(\huaG(\huaA), \huaG(\huaA^*); (L_0^*, L_1^*),(\huaL_0^*, \huaL_1^*))$ is a matched pair of strict Lie $2$-algebras if and only if $(\huaA, \huaA^*)$ is a strict pre-Lie $2$-bialgebra.  
\end{proof}

\begin{thm}\label{pf-3}
Let $\huaA=(A_0, A_{-1}, \dM, \cdot)$ be a strict pre-Lie $2$-algebra and $\huaA^*=(A^*_{-1},A^*_{0},\dM^*,\ast)$ be a strict pre-Lie $2$-algebra on its dual space $\huaA^*$. Then the following conditions are equivalent:
\begin{itemize}
	\item[\rm($1$)] $((\huaG(\huaA)\Join\huaG(\huaA^*), \huaG(\huaA), \huaG(\huaA^*), \omega)$ is a para-K\"ahler strict Lie $2$-algebra, where $\omega$ is given by \eqref{omega}; 
	\item[\rm($2$)] $(\huaG(\huaA), \huaG(\huaA^*); (L_0^*, L_1^*),(\huaL_0^*, \huaL_1^*))$ is a matched pair of strict Lie $2$-algebras; 
	\item[\rm($3$)] $(\huaA, \huaA^*; (\ad_0^*, \ad_1^*), (-R_0^*, -R_1^*), (\mathfrak{ad}_0^*, \mathfrak{ad}_1^*), (-\huaR_0^*, -\huaR_1^*))$ is a matched pair of strict pre-Lie $2$-algebras; 
	\item[\rm($4$)] $(\huaA, \huaA^*)$ is a strict pre-Lie $2$-bialgebra.
\end{itemize}
\end{thm}

 \subsection{Coboundary strict pre-Lie 2-bialgebras}
 For any 1-cochain $(r, \Phi) \in(A_0\otimes A_{-1} \oplus A_{-1}\otimes A_0)\oplus \Hom(A_0, A_{-1}\otimes A_{-1})$, we have: 
\begin{eqnarray*}
D(r, \Phi)&&=(\hat{\dM}+\bar{\dM}+\hat{\dM^\otimes})(r, \Phi)\\
&&=\hat{\dM^\otimes}r+\bar{\dM} r+\hat{\dM^\otimes}\Phi+\hat{\dM}\Phi+\bar{\dM}\Phi\\
&&=\dM^\otimes r+\bar{\dM} r-\dM^\otimes\circ \Phi-\Phi\circ \dM+\bar{\dM}\Phi.
\end{eqnarray*}
 Therefore, if $(\alpha_0, \alpha_1)=D(r, \Phi)$ for some 1-cochain $(r, \Phi)$, we must have
 \begin{eqnarray}
 \hat{\dM^\otimes}r&=&\dM^\otimes r=(\dM\otimes 1-1\otimes \dM)r=0,\\
 \bar{\dM}\Phi&=&0,\\
 \alpha_0(x)&=&\bar{\dM} r(x)+\hat{\dM^\otimes}\Phi(x)=(L_0\otimes 1+1\otimes \ad_0)r(x)-\dM^\otimes\circ \Phi(x),\\
 \alpha_1(a)&=&\bar{\dM} r(a)+\hat{\dM}\Phi(a)=(L_1\otimes 1+1\otimes \ad_1)r(a)-\Phi(\dM a).
 \end{eqnarray}
 
 Since we require $\bar{\dM}\Phi=0$, we choose $\Phi=\bar{\dM}\tau$ for some $\tau \in A_{-1}\otimes A_{-1}$.
 \begin{pro}\label{pf-1}
 If $\Phi=\bar{\dM}\tau$ for some $\tau \in A_{-1}\otimes A_{-1}$, then we have
 \begin{eqnarray}
 \label{2-coboundary-1}\alpha_0(x)&=&(L_0\otimes 1+1\otimes \ad_0)(r-\dM^\otimes\tau)(x),\\
 \label{2-coboundary-2}\alpha_1(a)&=&(L_1\otimes 1+1\otimes \ad_1)(r-\dM^\otimes\tau)(a).
 \end{eqnarray}
  \end{pro}
 \begin{proof}
 By $D^2=0$, we have $\hat{\dM^\otimes}\circ \bar{\dM}\tau+ \bar{\dM}\circ \hat{\dM^\otimes}\tau=0$. And $\hat{\dM^\otimes}\circ \bar{\dM}\tau=-\dM^\otimes\circ \bar{\dM}\tau$, $\bar{\dM}\circ \hat{\dM^\otimes}\tau=\bar{\dM}[(-1)^0\dM^\otimes\circ\tau]=\bar{\dM}\circ \dM^\otimes\tau$, which implies that $\bar{\dM}\circ \dM^\otimes\tau=\dM^\otimes\circ \bar{\dM}\tau$. Thus, we have 
 \begin{eqnarray*}
 \alpha_0(x)&=&\bar{\dM} r(x)+\hat{\dM^\otimes}\Phi(x)=\bar{\dM} r(x)-\dM^\otimes\circ\Phi(x)\\
 &=&\bar{\dM} r(x)-\dM^\otimes\circ \bar{\dM}\tau(x)=\bar{\dM} r(x)-\bar{\dM}\circ \dM^\otimes\tau(x)\\
 &=&\bar{\dM}(r-\dM^\otimes\tau)(x)=(L_0\otimes 1+1\otimes \ad_0)(r-\dM^\otimes\tau)(x).
 \end{eqnarray*}
 
 Also by $D^2=0$, we have $\hat{\dM}(\bar{\dM}\tau)+\bar{\dM}(\hat{\dM^\otimes}\tau)=0$. And $\hat{\dM}(\bar{\dM}\tau)(a)=-\bar{\dM}\tau(\dM a)$, $\bar{\dM}(\hat{\dM^\otimes}\tau)(a)=\bar{\dM}(\dM^\otimes\circ\tau)(a)$, which implies that $\bar{\dM}(\dM^\otimes\circ\tau)(a)=\bar{\dM}\tau(\dM a)$. Thus, we have
 \begin{eqnarray*}
 \alpha_1(a)&=&\bar{\dM} r(a)+\hat{\dM}\Phi(a)=\bar{\dM} r(a)-\Phi(\dM a)\\
 &=&\bar{\dM} r(a)-\bar{\dM}\tau(\dM a)=\bar{\dM} r(a)-\bar{\dM}(\dM^\otimes\circ\tau)(a)\\
 &=&\bar{\dM}(r-\dM^\otimes\tau)(a)=(L_1\otimes 1+1\otimes \ad_1)(r-\dM^\otimes\tau)(a).
 \end{eqnarray*}
 \end{proof}
 
 Let $A$ be a pre-Lie algebra and $r=\displaystyle\sum_{i}a_i\otimes b_i$ and we set 
 \begin{eqnarray*}
 &&r_{12}=\displaystyle\sum_{i}a_i\otimes b_i\otimes1; \quad r_{21}=\displaystyle\sum_{i}b_i\otimes a_i\otimes1;\\
 &&r_{13}=\displaystyle\sum_{i}a_i\otimes1\otimes b_i; \quad r_{23}=\displaystyle\sum_{i}1\otimes a_i\otimes b_i \in U(\huaG(A)),
 \end{eqnarray*}
 where $U(\huaG(A))$ is the universal enveloping algebra of the sub-adjacent Lie algebra $\huaG(A)$.
 
 \begin{defi}{\rm (\cite{Left-symmetric bialgebras})}
 A {\bf pre-Lie bialgebra} $(A, A^*, \alpha, \beta)$ is called coboundary if $\alpha$ is a 1-coboundary of $\huaG(A)$ associated to $L\otimes1+1\otimes \ad$, that is, there exists a $r\in A\otimes A$ such that
 \begin{eqnarray}\label{coboundary}
 \alpha(x)=(L_x\otimes 1+1\otimes \ad_x)r, \quad \forall~x\in A.
 \end{eqnarray}
 \end{defi}
 
 \begin{thm}{\rm (\cite{Left-symmetric bialgebras})}\label{pf-2}
 Let $A$ be a pre-Lie algebra and $r\in A\otimes A$. Then the map $\alpha$ defined by \eqref{coboundary} induces a pre-Lie algebra structure on $A^*$ such that $(A, A^*)$ is a pre-Lie bialgebra if and only if the following two conditions are satisfied:
\begin{itemize}
	\item[\rm(a)]$ [P(x\cdot y)-P(x)P(y)](r_{12}-r_{21})=0$ for any $x, y \in A$;
 \item[\rm(b)] $ Q(x)[[r, r]]=0,$
 \end{itemize}
  where $[[r, r]]$ is given by 
 $[[r, r]]=r_{13}\cdot r_{12}-r_{23}\cdot r_{21}+[r_{23}, r_{12}]-[r_{13}, r_{21}]-[r_{13}, r_{23}],$
 and $Q(x)=L_x\otimes1\otimes1+1\otimes L_x\otimes1+1\otimes1\otimes \ad_x$, $P(x)=L_x\otimes1+1\otimes L_x$ for any $x \in A$.
 \end{thm}

  \begin{thm}\label{pre-Lie 2-bialgebra}
 Let $(A_0, A_{-1}, \dM, \cdot)$ be a strict pre-Lie $2$-algebra, $r\in A_0\otimes A_{-1}\oplus A_{-1}\otimes A_0$ and $\tau\in A_{-1}\otimes A_{-1}$. Let $R=r-\dM^\otimes\tau=r-(\dM\otimes1+1\otimes \dM)\tau\in A_0\otimes A_{-1}\oplus A_{-1}\otimes A_0$. Then the maps $\alpha_0$ and $\alpha_1$ defined by \eqref{2-coboundary-1} and \eqref{2-coboundary-2} induce a pre-Lie $2$-algebra structure on $A^*$ such that $(A, A^*)$ is a pre-Lie 2-bialgebra if and only if for any $x+a, y+b\in A_0\oplus A_{-1}$, the following conditions are satisfied:
 \begin{itemize}
	\item[\rm(a)]$ [P((x+a)\ast (y+b))-P(x+a)P(y+b)](R-\sigma(R))=0$;
 \item[\rm(b)]$ Q(x+a)[[R, R]]_s=0$;
 \item[\rm(c)]$ \dM^\otimes r=0,$
 \end{itemize}
 where $\sigma:A_0\otimes A_{-1}\oplus A_{-1}\otimes A_0 \longrightarrow A_{-1}\otimes A_0\oplus A_0\otimes A_{-1}$ is the exchanging operator and $[[\cdot, \cdot]]_s$ is defined on the semi-product pre-Lie algebra $A_{0}\ltimes  A_{-1}$ given by \eqref{eq:semi-product}.
 \end{thm}
 \begin{proof}
Since $(\alpha_0, \alpha_1)=D(r, \Phi)$ is an exact cocycle, we only need to show that $(\alpha_0,\alpha_1)$ induces a strict pre-Lie $2$-algebra structure on $A^\ast_{-1}\oplus A^\ast_0$. By the definition of $(\alpha_0,\alpha_1)$, the later is equivalent to that $A^\ast_{-1}\ltimes A^\ast_0$ is a semi-product pre-Lie algebra.  Then the conclusion follows from Proposition \ref{pf-1} and Theorem \ref{pf-2}.
 \end{proof}
 
 \begin{defi}
 Let $\huaA=(A_0, A_{-1}, \dM, \cdot)$ be a strict pre-Lie $2$-algebra, $r\in A_0\otimes A_{-1}\oplus A_{-1}\otimes A_0$ and $\tau\in A_{-1}\otimes A_{-1}$. Let $R=r-\dM^\otimes\tau=r-(\dM\otimes1+1\otimes \dM)\tau\in A_0\otimes A_{-1}\oplus A_{-1}\otimes A_0$. Then the equations
  \begin{itemize}
 	\item[\rm(a)]$ R-\sigma(R)=0$;
 	\item[\rm(b)]$ [[R, R]]_s=0$;
 	\item[\rm(c)]$ \dM^\otimes r=0,$
 \end{itemize} 
   are called the {\bf $2$-graded classical Yang-Baxter Equations} (2-graded CYBEs) in the strict pre-Lie $2$-algebra $\huaA$.
 \end{defi}
 Let $\huaA=(A_0, A_{-1}, \dM, \cdot)$ be a strict pre-Lie $2$-algebra. If $R$ is a solution of  2-graded CYBEs, then the maps $\alpha_0$ and $\alpha_1$ defined by \eqref{2-coboundary-1} and \eqref{2-coboundary-2} induce a pre-Lie $2$-algebra structure on $A^*$ such that $(A, A^*)$ is a pre-Lie 2-bialgebra.

 \emptycomment{The simplest way to satisfy the three conditions of Theorem \ref{pre-Lie 2-bialgebra} is to assume that $R$ is symmetric, $[[R, R]]=0$ and $\dM^\otimes r=0$.

 We need to introduce the definition of $S$-equation in pre-Lie algebra $A$.
 \begin{defi}
 	Let $A$ be a pre-Lie algebra and $r\in A\otimes A$. Suppose $r$ is symmetric. Then 
 	$$[[r, r]]=-r_{12}\cdot r_{13}+r_{12}\cdot r_{23}+[r_{13}, r_{23}]=0$$ is called $S$-equation in $A$.
 \end{defi}

 \begin{defi}
 	Let $\huaA=(A_0, A_{-1}, \dM, \cdot)$ be a strict pre-Lie $2$-algebra, $r\in A_0\otimes A_{-1}\oplus A_{-1}\otimes A_0$ and $\tau\in A_{-1}\otimes A_{-1}$. Let $R=r-\dM^\otimes\tau=r-(\dM\otimes1+1\otimes \dM)\tau\in A_0\otimes A_{-1}\oplus A_{-1}\otimes A_0$, and $R$ is symmetric. The $S$-equation for $R$ in pre-Lie algebra $(A_0\oplus A_{-1}, \ast)$ together with $(\dM\otimes1-1\otimes \dM)R=0$ are called $S$-equation in the strict pre-Lie $2$-algebra $\huaA$, where $\ast$ is the pre-Lie algebra structure given by the following.
 \end{defi}

 	Let $\huaA=(A_0, A_{-1}, \dM, \cdot)$ be a strict pre-Lie $2$-algebra. Then we can get that $((A_0, \cdot), (A_{-1}, \circ), \dM, (\rho, \mu))$ is a crossed module of pre-Lie algebra, where $\circ$ in $A_{-1}$ is defined by $a\circ b=(\dM a)\cdot b=a\cdot(\dM b)$, for any $a, b \in A_{-1}$, $\rho, \mu:A_0 \longrightarrow gl(A_{-1})$ is given by 
 	$$\rho(x)a=x\cdot a, \mu(x)a=a\cdot x, \quad x\in A_0.$$
 	Then there is a pre-Lie algebra structure $\ast$ in $A_0\oplus A_{-1}$, and $\ast$ is given by
 	$$(x+a)\ast(y+b)=x\cdot y+\rho(x)b+\mu(y)a+a\circ b.$$}

Let $\huaA=A_0 \oplus A_{-1}$ be a graded vector spaces. For any $R\in A_0\otimes A_{-1}\oplus A_{-1}\otimes A_0$, define $R_0:\g_{-1}^\ast\rightarrow \g_0$ and $R_1:\g_{0}^\ast\rightarrow \g_{-1}$ by
\begin{equation}
	\left\langle x^*\otimes a^*+b^*\otimes y^*, R\right\rangle=\left\langle x^*\otimes a^*, R\right\rangle+\left\langle b^*\otimes y^*, R\right\rangle=\left\langle x^*, R_0(a^*)\right\rangle+\left\langle b^*, R_1(y^*)\right\rangle,
\end{equation}
for any $x^*, y^* \in A_0^*$, $a^*, b^* \in A_{-1}^*$.

 \begin{thm}
 	Let $\huaA=(A_0, A_{-1}, \dM, \cdot)$ be a strict pre-Lie $2$-algebra and $R\in A_0\otimes A_{-1}\oplus A_{-1}\otimes A_0$. If $R=\sigma(R)$, then $R$ is a solution of $2$-graded classical Yang-Baxter Equations in the strict pre-Lie $2$-algebra $\huaA$ if and only if $(R_0,R_1)$ is an  $\mathcal{O}$-operator on the subadjacent Lie 2-algebra $\huaG(\huaA)$ associated to the coregular representation  $(\huaA^\ast;L_0^\ast,L_1^\ast)$. 
 \end{thm}
 \begin{proof}
 	Let $\{e_i\}_{1\leq i\leq k}$ and $\{\frke_j\}_{1\leq j\leq l}$ be the basis of $A_0$ and $A_{-1}$ respectively, and denote by $\{e_i^*\}_{1\leq i\leq k}$ and $\{\frke_j^*\}_{1\leq j\leq l}$ the dual basis. Suppose $e_i\cdot e_j=\displaystyle\sum_{k}c_{ij}^ke_k$, $e_i\cdot \frke_j=\displaystyle\sum_{k}d_{ij}^k\frke_k$, $\frke_i\cdot e_j=\displaystyle\sum_{k}f_{ij}^k\frke_k$, $R=\displaystyle\sum_{i,j}a_{ij}(e_i\otimes \frke_j+\frke_j\otimes e_i)$ with $a_{ij}=a_{ji}$. Hence $R_0(\frke_i^*)=\displaystyle\sum_{k}a_{ik}e_k$ and $R_1(e_i^*)=\displaystyle\sum_{k}a_{ik}\frke_k$. Then we have 
 	\begin{eqnarray*}
 		[[R,R]]_s&=&-R_{12}\cdot R_{13}+R_{12}\cdot R_{23}+[R_{13},R_{23}]\\
 		&=&\displaystyle\sum_{i,j,k}\Big(\sum_{t,l}(-a_{kl}a_{tj}d_{tl}^i+a_{lk}a_{ti}d_{tl}^j+a_{lj}a_{ti}c_{tl}^k-a_{ti}a_{lj}c_{lt}^k)\frke_i\otimes\frke_j\otimes e_k\\
 		&& +\sum_{t,l}(-a_{lk}a_{it}f_{tl}^j+a_{lk}a_{tj}c_{tl}^i+a_{it}a_{lj}d_{lt}^k-a_{lj}a_{it}f_{tl}^k)\frke_j\otimes e_i\otimes \frke_k\\
 		&& +\sum_{t,l}(-a_{lk}a_{tj}c_{tl}^i+a_{lk}a_{it}f_{tl}^j+a_{lj}a_{it}f_{tl}^k-a_{it}a_{lj}d_{lt}^k)e_i\otimes\frke_j\otimes\frke_k\Big).
 	\end{eqnarray*}
 	Therefore, $R$ is a solution of $2$-graded classical Yang-Baxter Equations in the strict pre-Lie $2$-algebra $\huaA$ if and only if 
 \begin{eqnarray}
\label{eq:O1}   \sum_{t,l}(-a_{kl}a_{tj}d_{tl}^i+a_{lk}a_{ti}d_{tl}^j+a_{lj}a_{ti}c_{tl}^k-a_{ti}a_{lj}c_{lt}^k)&=&0, \\
 \label{eq:O2}   \sum_{t,l}(-a_{lk}a_{it}f_{tl}^j+a_{lk}a_{tj}c_{tl}^i+a_{it}a_{lj}d_{lt}^k-a_{lj}a_{it}f_{tl}^k)&=&0.
 \end{eqnarray}
On the other hand, we have 
 	\begin{eqnarray*}
 		R_0(L_0^*(R_0(\frke_i^*))\frke_j^*-L_0^*(R_0(\frke_j^*))\frke_i^*)-[R_0(\frke_i^*), R_0(\frke_j^*)]&=&\displaystyle\sum_{t,l,k}(-a_{it}a_{lk}d_{tl}^j+a_{jt}a_{lk}d_{tl}^i-a_{it}a_{jl}(c_{tl}^k-c_{lt}^k))e_k,\\
 		R_1(L_1^*(R_1(e_i^*))\frke_j^*-L_0^*(R_0(\frke_j^*))e_i^*)-[R_1(e_i^*), R_0(\frke_j^*)]&=&\displaystyle\sum_{t,l,k}(-a_{it}a_{lk}f_{tl}^j+a_{jt}a_{lk}c_{tl}^i-a_{it}a_{jl}(f_{tl}^k-d_{lt}^k))\frke_k.
 	\end{eqnarray*}
Thus,  $R_0(L_0^*(R_0(\frke_i^*))\frke_j^*-L_0^*(R_0(\frke_j^*))\frke_i^*)=[R_0(\frke_i^*), R_0(\frke_j^*)]$ if and only if \eqref{eq:O1} holds, and $R_1(L_1^*(R_1(e_i^*))\frke_j^*-L_0^*(R_0(\frke_j^*))e_i^*)=[R_1(e_i^*), R_0(\frke_j^*)]$ if and only if \eqref{eq:O2} holds. The conclusion follows.
 \end{proof}

 Let $(\rho_0,\rho_1)$ be a strict representation of the Lie 2-algebra $\huaG=(\g_0,\g_{-1}, \frkd, \frkl_2)$ on the 2-term complex of vector space $\huaV:V_{-1}\stackrel{\dM}{\longrightarrow}V_0$.
We view $\rho_0\oplus \rho_1$ a linear map from $\g_0\oplus \g_{-1}$ to $ \gl(V_0\oplus V_{-1})$ by
\begin{equation}
  (\rho_0\oplus\rho_1)(x+a)(u+m)=\rho_0(x)(u)+\rho_0(x)m+\rho_1(a)u.
\end{equation}
By straightforward computations, we have

\begin{lem}\label{lem:Ooperatorequi}{\rm (\cite{Sheng19})}
With the above notations,  $\rho_0\oplus \rho_1:\g_0\oplus \g_{-1}\longrightarrow \gl(V_0\oplus V_{-1})$ is a representation of $(\g_0\oplus \g_1,[\cdot,\cdot]_s)$ on $V_0\oplus V_1$. Furthermore, $(T_0,T_1)$ is an $\mathcal O$-operator on $\huaG$ associated to the representation $(\rho_0,\rho_1)$ if and only if \begin{itemize}
  \item[\rm(a)] $T_0+ T_1:V_0\oplus V_{-1}\longrightarrow \g_0\oplus \g_{-1}$ is an $\mathcal O$-operator on the Lie algebra $(\g_0\oplus \g_1,[\cdot,\cdot]_s)$ associated to the representation $\rho_0\oplus \rho_1$,
       \item[\rm(b)] $T_0\circ \dM=\frkd\circ T_1.$
 \end{itemize}
\end{lem}
 
 Let $(T_0,T_1)$ be an $\mathcal O$-operator on a Lie $2$-algebra $\huaG$ associated to the representation $(\rho_0,\rho_1)$. Define a degree $0$ multiplication $\cdot:V_i\otimes V_j\longrightarrow V_{i+j}$, $-1\leq i+j\leq 0$, on $\huaV$ by
    \begin{equation*}
      u\cdot v=\rho_0(T_0u)v,\quad u\cdot m=\rho_0(T_0u)m,\quad m\cdot u=\rho_1(T_1m)u.
    \end{equation*}  
      Then, by Proposition \ref{pro:RB-Lie2}, $(V_0,V_{-1},\dM,\cdot)$ is a strict pre-Lie  $2$-algebra. Let $T(\huaV)=\{(T_0u,T_1m)\mid u\in V_0,m\in V_{-1}\}$, and there is an induced pre-Lie  $2$-algebra structure on $T(\huaV )\subset\huaG$ given by 
 \begin{equation*}
      T_0u\cdot T_0v=T_0(\rho_0(T_0u)v),\quad T_0u\cdot T_1m=T_1(\rho_0(T_0u)m),\quad T_1m\cdot T_0u=T_0(\rho_1(T_1m)u).
    \end{equation*}  
 Let $\rho^*=(\rho_0^*, \rho_1^*)$ be the dual representation $\rho=(\rho_0, \rho_1)$. Then we have the semidirect product pre-Lie $2$-algebra $\bar{\huaG}=T(\huaV)\ltimes_{\rho^*,0}\huaV^*$, where $\bar{\huaG}_0=T(\huaV)_0\oplus V_{-1}^*$, $\bar{\huaG}_{-1}=T(\huaV)_{-1}\oplus V_0^*$ and $\bar{\dM}=\frkd\oplus\dM^*$. It is obvious that 
 $$T_0+T_1 \in V_0^*\otimes T(\huaV)_0\oplus V^*_{-1}\otimes T(\huaV)_{-1} \subset (\bar{\huaG}_{-1}\otimes\bar{\huaG}_0)\oplus(\bar{\huaG}_0\otimes\bar{\huaG}_{-1}).$$

 \begin{lem}\label{operator-2}{\rm(\cite{Left-symmetric bialgebras})}
 	Let $\huaG$ be a Lie algebra. Let $\rho:\huaG\longrightarrow\gl(V)$ be a representation of $\huaG$ and $\rho^*:\huaG\longrightarrow\gl(V^*)$ be its dual representation. Suppose that $T:V\longrightarrow\huaG$ is an $\mathcal{O}$-operator associated to $\rho$. Then $$r=T+\sigma(T)$$ is a symmetric solution of the classical Yang-Baxter Equations in the semi-product pre-Lie algebra $T(V)\ltimes_{\rho^*,0}V^*$, where $T(V)=\{T(u)\in\huaG| u\in V\}$ is the pre-Lie algebra with the pre-Lie operation defined by
 $$Tu\cdot Tv=T(\rho(Tu)v),\quad u,v\in V,$$
 and $T$ can be identified as an element in $T(V)\otimes V^*\subset(T(V)\ltimes_{\rho^*,0}V^*)\otimes(T(V)\ltimes_{\rho^*,0}V^*)$.
 \end{lem}
 
 \begin{thm}\label{solution-1}
 	Let $(T_0,T_1)$ be an $\mathcal O$-operator on a strict Lie $2$-algebra $\huaG$ associated to the strict representation $(\rho_0,\rho_1)$  on vector spaces $\huaV:V_{-1} \stackrel{\dM}{\longrightarrow} V_{0}$ and $\rho^*=(\rho_0^*, \rho_1^*)$ be its dual representation. Then $$R=T_0+T_1+\sigma(T_0+T_1)$$
 	is a solution of $2$-graded $2$-graded classical Yang-Baxter Equations in the semidirect product pre-Lie $2$-algebra $T(\huaV)\ltimes_{\rho^*,0}\huaV^*$. 
 \end{thm}	
 \begin{proof}
 	It is obvious that $(\rho_0\oplus\rho_1)^*=\rho_0^*\oplus\rho_1^*:\g_0\oplus \g_{-1}\longrightarrow\gl(V_0^*\oplus V_{-1}^*)$. If $(T_0,T_1)$ is an $\mathcal{O}$-operator on $\huaG$ associated to the representation $\rho=(\rho_0,\rho_1)$, by Lemma \ref{lem:Ooperatorequi} and Lemma \ref{operator-2},  $R=T_0+T_1+\sigma(T_0+T_1)$ is a  symmetric solution of the classical Yang-Baxter Equations in the semi-product pre-Lie algebra $T(\huaV)\ltimes_{\rho_0^*\oplus\rho_1^*,0}(V_{-1}^*\oplus V_0^*)$ and $T_0\circ \dM=\frkd\circ T_1$. Note that the semidirect product pre-Lie algebra $T(\huaV)\ltimes_{\rho_0^*\oplus\rho_1^*,0}(V_{-1}^*\oplus V_0^*)$ is the same as the semidirect product pre-Lie algebra $\bar{\huaG}_0\ltimes\bar{\huaG}_{-1}$ given by Proposition \ref{pro:semi-product}. Furthermore, $T_0\circ \dM=\frkd\circ T_1$ if and only if $(\bar{\dM}\otimes1-1\otimes\bar{\dM})(T_0+T_1)=0$. Thus, if $(T_0,T_1)$ is an $\mathcal{O}$-operator on $\huaG$ associated to the representation  $\rho=\rho_0,\rho_1$,  $R=(T_0+T_1)+\sigma(T_0+T_1)$ is a solution of  $2$-graded classical Yang-Baxter Equations in the semidirect product pre-Lie $2$-algebra $T(\huaV)\ltimes_{\rho^*,0}\huaV^*$.
 \end{proof}

 \begin{cor}
 Let $\huaA=(A_0,A_{-1},\dM,\cdot)$ be a strict pre-Lie $2$-algebra. Let $\{e_i\}_{1\leq i\leq k}$ and $\{\frke_j\}_{1\leq j\leq l}$ be the basis of $A_0$ and $A_{-1}$ respectively, and denote by $\{e_i^*\}_{1\leq i\leq k}$ and $\{\frke_j^*\}_{1\leq j\leq l}$ the dual basis. Then 
 \begin{equation}
 	R=\displaystyle\sum_{i=1}^{k}(e_i\otimes e_i^*+e_i^*\otimes e_i)+\displaystyle\sum_{j=1}^{l}(\frke_j\otimes \frke_j^*+\frke_j^*\otimes \frke_j)
 \end{equation}
 is a solution of $2$-graded classical Yang-Baxter Equations in $\huaA\ltimes_{L^*,0}\huaA^*$ with $L^*=(L_0^*, L_1^*)$.
 \end{cor}
 \begin{proof}
 	Let $\huaV=\huaA$, $\rho=(\rho_0,\rho_1)=(L_0, L_1)$ and $(T_0,T_1)=(\mathrm{id}_{A_0},\mathrm{id}_{A_1})$ in Theorem \ref{solution-1}. It is obvious that $(\mathrm{id}_{A_0},\mathrm{id}_{A_1})$ is an $\mathcal{O}$-operator on $\huaG(\huaA)$ associated to the representation $(L_0, L_1)$. Then $$R=T_0+T_1+\sigma(T_0+T_1)=\displaystyle\sum_{i=1}^{k}(e_i\otimes e_i^*+e_i^*\otimes e_i)+\displaystyle\sum_{j=1}^{l}(\frke_j\otimes \frke_j^*+\frke_j^*\otimes \frke_j)$$ is a solution of $2$-graded classical Yang-Baxter Equations in $\huaA\ltimes_{L^*,0}\huaA^*$.
 \end{proof}

\noindent
{\bf Acknowledgements.} This research is supported by  NSFC (12201068, 12371029, W2412041) and the National Key Research and Development Program of China (2021YFA1002000).

 \end{document}